%% file: main.tex
\documentclass[12pt]{amsart}
\usepackage{hyperref}
\hypersetup{
    colorlinks=true, 
     citecolor=teal, 
     linkcolor=blue, 
    }

\usepackage[utf8]{inputenc}
\usepackage[english]{babel}
\usepackage[T1]{fontenc}
\usepackage{amsmath}
\usepackage{amsfonts}
\usepackage{amssymb}
\usepackage{subfigure}
\usepackage{todonotes}
\usepackage{mathtools}

\usepackage[toc,page]{appendix} 
\usepackage{amsthm} 
\usepackage{caption}
\usepackage{tikz-cd}
\usepackage{IEEEtrantools} 
\usepackage{bm} 
\usepackage{graphicx}
\usepackage[percent]{overpic}
\usepackage{textcomp} 

\usepackage{cleveref}
\let\cref\Cref

\usepackage[utf8]{inputenc}
\usepackage{amsmath}
\usepackage{amsfonts}
\usepackage{amssymb}
\usepackage{tikz}
\usetikzlibrary{decorations.markings}
\usetikzlibrary{patterns}

\usepackage{transparent}

\usepackage{pgfplots}
\usetikzlibrary {math}
\usepgfplotslibrary{colormaps}
\usepgfplotslibrary{fillbetween}

\usepackage[utf8]{inputenc}
\usepackage{amsmath}
\usepackage{amsfonts}
\usepackage{amssymb}
\usepackage{tikz}
\usetikzlibrary{decorations.markings}
\usepackage{tikz}
\usepackage{graphicx}
\usepackage[utf8]{inputenc}
\usepackage{amsmath}
\usepackage{amsfonts}
\usepackage{amssymb}
\usepackage{tikz}
\usetikzlibrary{decorations.markings}
\usepackage{pgfplots}
\usetikzlibrary{patterns}
\makeatletter
\pgfdeclarepatternformonly[\LineSpace]{my north east lines}{\pgfqpoint{-1pt}{-1pt}}{\pgfqpoint{\LineSpace}{\LineSpace}}{\pgfqpoint{\LineSpace}{\LineSpace}}%
{
    \pgfsetcolor{\tikz@pattern@color}
    \pgfsetlinewidth{0.4pt}
    \pgfpathmoveto{\pgfqpoint{0pt}{0pt}}
    \pgfpathlineto{\pgfqpoint{\LineSpace + 0.1pt}{\LineSpace + 0.1pt}}
    \pgfusepath{stroke}
}

\pgfdeclarepatternformonly[\LineSpace]{my north west lines}{\pgfqpoint{-1pt}{-1pt}}{\pgfqpoint{\LineSpace}{\LineSpace}}{\pgfqpoint{\LineSpace}{\LineSpace}}%
{
    \pgfsetcolor{\tikz@pattern@color}
    \pgfsetlinewidth{0.4pt}
    \pgfpathmoveto{\pgfqpoint{0pt}{\LineSpace}}
    \pgfpathlineto{\pgfqpoint{\LineSpace + 0.1pt}{-0.1pt}}
    \pgfusepath{stroke}
}
\makeatother

\newdimen\LineSpace
\tikzset{
    line space/.code={\LineSpace=#1},
    line space=10pt
}

\usepackage{tkz-euclide}

\numberwithin{equation}{section}

\newtheorem{thm}{Theorem}[section]
\newtheorem{prop}[thm]{Proposition}

\newtheorem{cor}[thm]{Corollary}
\newtheorem{lem}[thm]{Lemma}

\theoremstyle{plain}

\newtheorem*{theorem*}{Theorem}

\theoremstyle{definition}
\newtheorem{defn}[thm]{Definition}
\newtheorem{claim}[thm]{Claim}

\theoremstyle{remark}
\newtheorem{rem}[thm]{Remark}
\theoremstyle{example}

\newcommand{\RR}{\mathbb{R}}

\newcommand{\bound}{\partial_T}
\newcommand{\hull}{\textup{Hull}}

\begin{document}
\title{Embedding products of trees into higher rank}
\author[Oussama Bensaid]{Oussama Bensaid}
\author[Thang Nguyen]{Thang Nguyen}

\address{Max Planck Institute for Mathematics, Vivatsgasse 7, 53111 Bonn, Germany}
\email{bensaid@mpim-bonn.mpg.de}

\address{Department of Mathematics, Florida State University,    Tallahassee, FL, 32304}
\email{tqn22@fsu.edu}

\thanks{TN is partially supported by Simons Travel Support for Mathematicians grants MPS-TSM-00002547.}
\maketitle
\begin{abstract}
    We show that there exists a quasi-isometric embedding of the product of $n$ copies of $\mathbb{H}_{\mathbb{R}}^2$ into any symmetric space of non-compact type of rank $n$, and there exists a bi-Lipschitz embedding of the product of $n$ copies of the $3$-regular tree $T_3$ into any thick Euclidean building of rank $n$ with co-compact affine Weyl group. This extends a previous result of Fisher--Whyte. The proof is purely geometrical, and the result also applies to the non Bruhat--Tits buildings.
\end{abstract}
\maketitle
\section{Introduction}
Symmetric spaces of non-compact type and Euclidean buildings are important classes of non-positively curved metric spaces. They possess large symmetry groups and structures that often distinguish them from other spaces, and also from one another, even if one only considers the coarse geometry. The latter approach is part of Gromov's program to classify spaces and groups from their coarse geometry, and was partly motivated by a remarkable theorem of Mostow \cite{mostow1973strong}. Some of the well-known theorems in this direction are by Pansu \cite{Pansu}, Schwartz \cite{schwartz1995quasi,Schwartz1996}, Kleiner--Leeb \cite{kleiner1997rigidity}, Eskin--Farb \cite{eskin1997quasi}, Eskin \cite{eskin1998quasi}, and Drutu \cite{drutu2000quasi}.

\medskip \noindent
Besides distinguishing these spaces, it is also interesting to study their relationships, especially which space is a totally geodesic subspace of another. While this question can be answered satisfactorily from the classification of semi-simple Lie groups and Lie triple systems, its coarse version is much more subtle. In other words, one might ask whether one space can be quasi-isometrically embedded into another one, or whether there is an obstruction to the existence of such an embedding. While there are many examples of isometric and quasi-isometric embeddings between rank one symmetric spaces, which can be constructed by a result of Bonk--Schramm \cite{bonk2011embeddings}, and examples of quasi-isometric embeddings from rank one into higher rank by Brady--Farb \cite{brady1998filling}, see also \cite{leebcharac, leuzinger2003bi} for Euclidean buildings, examples of embeddings between two spaces of equal and higher rank are very limited.

\medskip \noindent
Recently, Fisher and Whyte \cite{fisher2018quasi} gave a sufficient condition for the existence of a quasi-isometric embedding between two symmetric spaces of non-compact type of equal rank. This condition is formulated in terms of the existence of a linear map between Cartan subalgebras preserving kernel of roots. A quasi-isometric embedding induced by this map is called an $AN$-map. They also provided examples of embeddings when their condition is held. In particular, they constructed a quasi-isometric embedding from the product of $n$ copies of the real hyperbolic plane into the symmetric spaces of $\textup{SL}_{n+1}(\mathbb{R})$ and $\textup{Sp}_{2n}(\mathbb{R})$. In \cite{nguyen2021quasi}, the second author gave a splitting decomposition of embeddings. Namely, any embedding between spaces of equal rank is close to a product of embeddings into irreducible targets. He further gave examples of quasi-isometric embeddings in rank $2$ which are not $AN$-maps. In this paper, we generalize this approach by showing the following.
\begin{thm}\label{mainthm}(see \cref{mainthm general} for a more general statement)
\noindent
\begin{itemize}
        \item[(1)] If $X$ is a thick Euclidean building of rank $n$ with co-compact affine Weyl group, there exists a bi-Lipschitz embedding $T_3 \times \dots \times T_3 \to X$ of the product of $n$ copies of the 3-regular tree into $X$.
        \item[(2)] If $X$ is a symmetric space of non-compact type of rank $n$, there exists a quasi-isometric embedding $\mathbb{H}_{\mathbb{R}}^2 \times \dots \times \mathbb{H}_{\mathbb{R}}^2 \to X$ of the product of $n$ copies of the real hyperbolic plane into $X$.
    \end{itemize} 
\end{thm}
\noindent
Our approach can be regarded as a geometric $AN$-map, which is much more flexible than the one introduced by Fisher--Whyte, especially for Euclidean buildings. In particular, the embeddings of products of trees into thick Euclidean buildings also hold for the exotic ones. For example, the ones of type $\Tilde{A_2}$ \cite{zbMATH03944686, van1987non, barre2000immeubles}, and the ones whose Weyl group does not come from a root system  \cite{hitzelberger2010non,berenstein2012affine}.

\medskip\noindent
Since $T_3$ embeds quasi-isometrically into $\mathbb{H}_{\mathbb{R}}^2$, it follows that the product of $n$ copies of $T_3$ embeds quasi-isometrically into any symmetric space of non-compact type of rank $n$. Moreover, by combining \cref{mainthm} with the quasi-isometric embeddings of Gromov-hyperbolic groups into products of binary trees \cite{buyalo2007embedding}, we get exotic quasi-isometric embeddings of Gromov hyperbolic groups into Euclidean buildings and symmetric spaces.
\begin{cor}
    Every Gromov hyperbolic group $G$ admits a quasi-isometric embedding into any thick Euclidean building with co-compact affine Weyl group or symmetric space of non-compact type of rank $n+1$, where $n$ is the topological dimension of $\partial_{\infty}G$.
\end{cor}
\noindent 
For example, $\mathbb{H}_{\mathbb{R}}^{n}$ embeds quasi-isometrically into any such Euclidean building of rank $n$. By combining with the quasi-isometric embeddings $\mathbb{H}_{\mathbb{R}}^n \to (\mathbb{H}_{\mathbb{R}}^2)^{n-1}$ of \cite{brady1998filling} instead, we get 
\begin{cor}
    For any $n \geq 1$, $\mathbb{H}_{\mathbb{R}}^{n+1}$ embeds quasi-isometrically into any symmetric space of non-compact type of rank $n$.
\end{cor}
\noindent
Finally, let us note that the quasi-isometric embedding of the product of copies of $\mathbb{H}_{\mathbb{R}}^2$ into a symmetric space can also obtained by an $AN$-map. This was pointed out to us by Yves Benoist, and we give a proof in the Appendix.

\subsection*{Main result} 
\noindent
\\We refer to \cref{background} for the background material. Let $X$ be a Euclidean building or a symmetric space of non-compact type of rank $n$, and let $F_0$ be a fixed apartment/maximal-flat. Given a wall $H \subset F_0$, i.e.\ a singular $(n-1)$-dimensional flat, we consider its cross section $CS(H)$, which can be seen as the set of $(n-1)$-flats in $X$ which are parallel to $H$, and we define a projection map $X \to CS(H)$ (see \cref{section background preliminary}). This projection map is defined by considering a suitable point $\eta$ in the boundary of $H$, and assigning to each $x \in X$ the unique $(n-1)$-flat parallel to $H$ to which the geodesic ray $[x, \eta)$ is strongly asymptotic, see \cref{proj}. We endow the cross sections with the Hausdorff distance. This projection map is a variation of the projection onto the space of strong asymptotic classes introduced by Leeb \cite{leebcharac}.

\medskip\noindent
If $\Delta$ is a subset of $\bound F_0$, we denote by $X_\Delta$ the union of all apartments/maximal-flats in $X$ that contain $\Delta$ in their boundary at infinity. 

\medskip\noindent
We show that there exist walls $H_1,\dots,H_n$ in $F_0$, projection maps $\pi_i : X \to CS(H_i)$, and $\Delta \subset \bound F_0$ a suitable union of chambers such that the following map
$$\pi : X_{\Delta} \to CS(H_1) \times \dots \times CS(H_n),$$ 
which is the restriction of the product map $\pi_1 \times \dots \times \pi_n$ to $X_{\Delta}$, satisfies the following. 
\begin{thm}\label{mainthm general}
\noindent
    \begin{itemize}
        \item[(1)] If $X$ is a Euclidean building, $\pi$ is a bi-Lipschitz map. Moreover, the inclusion of $X_{\Delta}$, equipped with the path-metric, in $X$ is a bi-Lipschitz embedding, therefore it induces a bi-Lipschitz embedding from $CS(H_1) \times \dots \times CS(H_n)$ into $X$.
        \item[(2)] If $X$ is a symmetric space of non-compact type, $\pi$ is a quasi-isometry. Moreover, the inclusion of some $\delta$-neighborhood of $X_{\Delta}$, equipped with the path-metric, in $X$ is a quasi-isometric embedding, therefore it induces a quasi-isometric embedding from $ CS(H_1) \times \dots \times CS(H_n)$ into $X$.
    \end{itemize}
\end{thm}
\noindent 
The target space is equipped with the $L^1$ product metric. We refer to \cref{section proof main result} for the construction of the walls $H_i$ and $\Delta$. \cref{mainthm} is an immediate consequence of \cref{mainthm general}:
\begin{proof}[Proof of \cref{mainthm}]
    $(1)$ If $X$ is a thick Euclidean building of rank $n$ with co-compact affine Weyl group, then all the cross sections are thick Euclidean buildings of rank 1, i.e.\ thick metric trees \cite{leebcharac}. Therefore there exists a bi-Lipchitz embedding $T_3 \to CS(H_i)$. 
    \\ $(2)$ When $X$ is a symmetric space of non-compact type, the cross sections are rank one symmetric space of non-compact type \cite[Chap.\ 2.20]{eberlein1996geometry}, therefore $\mathbb{H}_{\mathbb{R}}^2 \to CS(H_i) $ isometrically.
\end{proof}
\begin{rem}
    The constants of the bi-Lipchitz/quasi-isometric embeddings do not depend on the apartment/maximal-flat $F_0$ we started with. Moreover, $F_0$ is contained in the image of such embeddings.
\end{rem}
\noindent
Let us note that when $X$ is a Euclidean building (resp.\ a symmetric space), the fact that the inclusion of $X_\Delta$ (resp.\ its $\delta$-neighbordhood), in $X$ is a bi-Lipschitz (resp.\ quasi-isometric) embedding, is a general result, as shown in step $4$ of the proof of \cref{mainthm general}, and is true when $\Delta$ is any union of chambers in the boundary of some fixed apartment/maximal-flat $F_0$. For example, if $\Delta$ consists of only one chamber, then $X_\Delta =X$. If $\Delta = \bound F_0$, then $X_\Delta = F_0$. Obviously, the more chambers we add to $\Delta$, the smaller $X_\Delta$ becomes. The subset $\Delta$ in \cref{mainthm general} is the ``smallest possible'' for which the map $\pi$ is injective (resp.\ quasi-injective).
\subsection*{About the proof}\label{subsect overview}

The proof of Theorem \ref{mainthm general} will be done in four steps. If $X$ is a Euclidean building (resp.\ a symmetric space), we will start by showing that $\pi$ is a bi-Lipschitz embedding when restricted to a flat containing $\Delta$ at infinity. Then, in step $2$, which represents the core of the proof, we show the general case, i.e.\ that it is a bi-Lipschitz (resp.\ quasi-isometric) embedding. In step $3$, we show that it is surjective (resp.\ quasi-surjective), and finally that the inclusion of $X_{\Delta}$ (resp.\ a thickening of $X_{\Delta}$), equipped with the path metric, in $X$ is a bi-Lipschitz (resp.\ quasi-isometric) embedding. 

\subsection*{Organisation of the paper}

We start in \cref{background} by recalling briefly the parallel sets and cross sections in symmetric spaces and Euclidean buildings. After giving some preliminary lemmas in \cref{subsec prelim}, we define the projection map onto a cross section, to which \cref{subsect proj map} is devoted. We recall the generalized Iwasawa decomposition for semi-simple Lie groups in \cref{section Iwasawa}. In \cref{section max distrib vertices}, we define the maximally distributed vertices in a spherical Coxeter complex. Finally, the main result is proved in \cref{section proof main result}.

\subsection*{Acknowledgements}
We thank Yves Benoist for pointing out to us the $AN$-map for symmetric spaces. We are grateful to the Max-Planck Institute for Mathematics in Bonn for its financial support. The second author thanks Institut des Hautes \'Etudes Scientifiques, Institut Henri Poincar\'e and LabEx CARMIN for their support and hospitality.
\section{Background and preliminary results}\label{section background preliminary}
\subsection{Background}\label{background}
We recall that if $X$,$Y$ are two metric spaces, and $f:X \to Y$ a map,
\begin{itemize}
    \item[$(i)$] $f$ is a \textit{bi-Lipschitz embedding} if there exists $\lambda \geq 1$ such that for any $x,x' \in X$, 
$$\frac{1}{\lambda} d_X(x,x') \leq d_Y(f(x),f(x')) \leq \lambda d_X(x,x').$$
If moreover $f$ is surjective, it is called a \textit{bi-Lipschitz equivalence.}
    \item[$(ii)$] $f$ is a \textit{quasi-isometric embedding} if there exist $\lambda \geq 1$ and $C \geq 0$ such that for any $x,x' \in X$, 
$$\frac{1}{\lambda} d_X(x,x') -C \leq d_Y(f(x),f(x')) \leq \lambda d_X(x,x') +C.$$
If moreover there exists $M \geq 0$ such that for any $x \in X$ there exists $y \in Y$ such that $d_Y(f(x),y) \leq M$, $f$ is called a \textit{quasi-isometry.}
\end{itemize}

\medskip\noindent
We refer to \cite{bridson2013metric} for the background material on CAT(0) spaces, and to \cite{eberlein1996geometry},\cite{kleiner1997rigidity}, and \cite{leebcharac} for symmetric spaces and Euclidean buildings.

\medskip\noindent
Let $X$ is a complete CAT(0) space. We denote by $\partial X$ its visual boundary, and we equip it with the \textit{angular metric} $\angle$ defined, for any $\xi,\eta \in \partial X$, as
$$ \, \angle(\xi,\eta) = \sup_{x\in X}   \angle_x(\xi,\eta).$$
Let $X$ be a Euclidean building or a symmetric space of non-compact type. We recall that its visual boundary $\partial X$, when equipped with the angular metric $\angle$, inherits a spherical building structure, and we denote it by $\bound X$. Let us note that we consider the angular metric $\angle$ and not the Tits metric on the boundary, i.e.\ the associated length metric\footnote{When $\textup{rank}(X) \geq 2$, these two metrics on $\partial X$ coincide}.

\medskip\noindent
The apartments of $\bound X$ are endowed with a structure of a spherical Coxeter complex \cite[sect.\ 3.1]{kleiner1997rigidity}. A \textit{spherical Coxeter complex} is a unit sphere $S$ with a finite Weyl group $W < \textup{Isom}(S)$ generated by reflections at \textit{walls}, i.e.\ totally geodesic subspheres of codimension $1$. A \textit{singular sphere} $s \subset S$ is an intersection of walls.

\medskip\noindent
The apartments in $\bound X$ correspond to boundaries of apartments/maximal-flats in $X$. Each apartment/maximal-flat $F \subset X$ is endowed with a structure of a Euclidean Coxeter complex \cite[sect.\ 4.1]{kleiner1997rigidity}. A \textit{Euclidean Coxeter complex} is a Euclidean space $E$ with an affine Weyl
group $W_{aff} < \textup{Isom}(E)$ generated by reflections at \textit{walls}, i.e.\ affine subspaces of
codimension $1$, so that the image of $W_{aff}$ in $\textup{Isom}(\bound E)$ is a finite reflection group. We call \textit{flat} any totally geodesic Euclidean subspace of $X$. A \textit{singular flat} in $X$ is an intersection of walls. Finally, a \textit{singular half-space} in $X$ is a half apartment/maximal-flat bounded by a wall.
\begin{rem}
    In the rest of the paper, by abuse of language, if $X$ is a Euclidean building, we will also say maximal-flat to refer to its apartments.
\end{rem}
\noindent
Let $s$ be a singular sphere in $\bound X$. The \textit{parallel set} $P(s)$ of $s$ is the union of all flats with boundary $s$. $P(s)$ is a convex subset of $X$ and is isometric to the product 
$$ P(s) = \RR ^{\textup{dim}(s)+1} \times CS(s).$$
$CS(s)$ is called the \textit{cross section} of $s$, and it can be seen as the set of flats with boundary $s$. When $X$ is a Euclidean building (resp.\ a symmetric space of non-compact type) of rank $n$, $CS(s)$ is a Euclidean building (resp.\ a symmetric space of non-compact type) of rank $n-\textup{dim}(s)-1$, see \cite[Chap.\ 2.20]{eberlein1996geometry},\cite[Sect.\ 3]{leebcharac},\cite[Sect.\ 4.8]{kleiner1997rigidity}. If $F$ is a flat such that $\bound F = s$, we define $P(F) := P(s)$. 
\subsection{Preliminary results}\label{subsec prelim}
Unless stated otherwise, $X$ is either a symmetric space of non-compact type or a Euclidean building. Symmetric spaces are supposed of non-compact type.

\medskip\noindent
If $x \in X$ and $\eta \in \bound X$, we denote by $[x,\eta)$ the geodesic ray from $x$ to $\eta$. If $A \subset \bound X$, we denote by $[x, A)$ the cone $\bigcup\big\{[x,a), a \in A \big\}$.

\noindent
If $A \subset X$ and $r \geq 0$, we denote by $N_r(A) = \{x \in X \textup{ such that } d_X(x,A) \leq r\}$.
\begin{lem}\label{parallel of interior pt}
    Let $X$ be a symmetric space or a Euclidean building, $\alpha$ a geodesic in $X$ with endpoints $\{\eta,\hat{\eta}\}$, and $F$ a flat containing $\alpha$. If $\eta$ is an interior point of a top-dimensional cell of $\bound F$, then $P(\alpha) = P(F)$. In other words, $P\left(\{\eta, \hat{\eta} \}\right) = P(\bound F)$.
\end{lem}
\begin{proof}
    Since $\alpha \subset F$, then $P(F) \subset P(\alpha)$.
    \\ Let $x \in P(\alpha)$, and let us denote by $c$ the top-dimensional cell of $\bound F$ containing $\eta$ in its interior. Let $\hat{c}$ be its opposite such that $\hat{\eta} \in \hat{c}$. Consider a maximal-flat $E$ in $X$ containing $\eta, \hat{\eta}$, so $c, \hat{c} \subset \bound E$ because  $\eta,\hat{\eta}$ are interior points. Since $s$ is the unique singular sphere spanned by $\eta$ and $\hat{\eta}$, $s \in \bound E$ and $x$ is contained in a flat $F' \subset E$ with boundary $s$.\qedhere
\end{proof}

\medskip\noindent
We say that a geodesic ray $\gamma$ is \textit{strongly asymptotic} to a subset $A \subset X$ if $d_X(\alpha(t),A) \xrightarrow[+\infty] \, 0 $.
\begin{lem}\label{entering parallel}
    Let $X$ be a symmetric space or a Euclidean building, $s$ a singular sphere in $\bound X$, and $\eta$ an interior point of a top-dimensional cell of $s$. For any $x \in X$, $[x,\eta)$ is strongly asymptotic to $P(s)$. Moreover, if $X$ is a Euclidean building, there exists $T \geq 0$ such that for $t \geq T$, $[x,\eta)(t) \in P(s)$.
\end{lem}
\begin{proof}
    Let $\hat{\eta}$ be the opposite of $\eta$ in $s$. 
    \\• If $X$ is a symmetric space, let $\alpha$ be a geodesic with endpoints $\eta, \hat{\eta}$. Consider the generalized Iwasawa decomposition, see \cref{iwasawa}, $\textup{Isom}_0(X) = K A_\eta N_\eta$ with respect to $\eta$ and a point in $\alpha$. Let $\gamma$ be the geodesic containing $x$ and $\eta$ at $+ \infty$, then there exists $a \in A_\eta$ and $n \in N_\eta$ such that $ \gamma = a n \alpha$. Note that $a \alpha$ is parallel to $\alpha$ so $a \alpha \subset P\left(\{\eta, \hat{\eta} \}\right) = P(s) $. Moreover, for any $t \in \mathbb{R}$
    $$ d_X(\gamma(t),a \alpha(t)) = d_X(an \alpha(t),a \alpha(t)) = d_X(n \alpha(t), \alpha(t)) \xrightarrow[+\infty] \, 0. $$
    • If $X$ is a Euclidean building: we denote the geodesic ray $[x,\eta)$ by $\gamma$. By the angle rigidity axiom \cite[Sect.\ 4.1.2]{kleiner1997rigidity}, $\angle_{\gamma(t)}(\eta,\hat{\eta})$ takes only finitely many values, and since $\angle_{\gamma(t)}(\eta,\hat{\eta}) \xrightarrow[+\infty] \, \angle(\eta,\hat{\eta})=\pi$, there exists $T \geq 0$ such that if $t\geq T$ then $\angle_{\gamma(t)}(\eta,\hat{\eta}) = \pi$, i.e.\ $\gamma(t)$ is in a geodesic joining $\eta$ and $\hat{\eta}$. Therefore, for $t \geq T$, $\gamma(t) \in P(\{ \eta, \hat{\eta} \}) = P(s)$ by Lemma \ref{parallel of interior pt}.\qedhere
\end{proof}
\noindent
An immediate consequence of \cref{entering parallel} is the following.
\begin{cor}\label{uniqueness projection}
If $X$ be a symmetric space or a Euclidean building, $s$ is a singular sphere in $\bound X$, and $\eta$ an interior point of a top-dimensional cell of $s$, then for any $x \in X$ there exists a unique flat $H$ with boundary $s$ to which $[x,\eta)$ is strongly asymptotic.
\end{cor}
\subsection{The projection map}\label{subsect proj map}
\cref{uniqueness projection} allows us to define a projection onto the cross section of a singular sphere via an interior point. 
\begin{defn}\label{proj}
    Let $s$ be a singular sphere in $\bound X$, and $\eta$ an interior point of a top-dimensional cell of $s$. We define the projection via $\eta$, $\pi : X \to CS(s)$, by assigning to $x$ the unique flat with boundary $s$ to which $[x,\eta)$ is strongly asymptotic.
\end{defn}
\begin{rem}
    Note that $\pi$ depends of the top-dimensional cell of $s$, but does not depend on the choice of its interior point.
\end{rem}
\noindent
We endow $CS(s)$, viewed as the set of all flats with boundary $s$ (hence all parallel), with the Hausdorff distance.
\begin{lem}
   The map $\pi$ is 1-Lipschitz.
\end{lem}
\begin{proof}
    Let $x,x' \in X$, and $\gamma=[x,\eta)$ and $\gamma'=[x',\eta)$. For any $t \geq 0$,
    $$d(\pi(x),\pi(x')) \leq d_X(\pi(x),\gamma(t)) + d_X(\gamma(t),\gamma'(t))+ d_X(\gamma'(t),\pi(x')). $$
    By convexity of the distance function, $d_X(\gamma(t),\gamma'(t)) \leq d_X(x,x')$, so when $t \to +\infty$ we get the result. \qedhere
\end{proof}
\noindent
\begin{rem}
    As mentioned in the introduction, this projection is a variation of the projection onto the space of strong asymptotic classes introduced by Leeb \cite{leebcharac}.
\end{rem}
\noindent
We end this section with some useful lemmas related to this projection map. Let us first recall that if $X$ is a CAT(0) space, $x_0 \in X$, and $\eta \in \partial X$, the Busemann function with respect to $x_0$ and $\eta$ is the map $b : X \to \mathbb{R}$ such that for any $x \in X$,
$$ b(x) = \displaystyle\lim_{t \to + \infty} d_X([x_0,\eta)(t),x)-t.$$
\begin{lem}\label{strong asymptot is convex}
    Let $X$ be a complete CAT(0) space, $x \in X$, $F$ a flat in $X$, and $\eta_1, \eta_2 \in \bound F$ such that $\angle(\eta_1,\eta_2)<\pi$.
    If $[x,\eta_1)$ is strongly asymptotic to $F$ and $\angle_x(\eta_1,\eta_2) = \angle(\eta_1,\eta_2)$, then for any interior point $\xi$ of $\overline{\eta_1 \eta_2}$, $[x,\xi)$ is strongly asymptotic to $F$.
\end{lem}
\begin{proof}
$\angle_x(\eta_1,\eta_2) = d_T(\eta_1,\eta_2)<\pi$ implies that $[x,\overline{\eta_1 \eta_2})$ is a flat sector, see \cite[Chap.\ 2 Cor.\ 9.9]{bridson2013metric}. Let $\varepsilon>0$, there exists $x_1 \in [x,\eta_1)$ such that $d_X(x_1,F) \leq \varepsilon$. By convexity of the distance function to $F$, $[x_1,\overline{\eta_1 \eta_2}) \subset N_\varepsilon (F)$. Since $[x,\overline{\eta_1 \eta_2})$ is a flat sector, for any interior point $\xi$ of $\overline{\eta_1 \eta_2}$, $[x,\xi)$ enters in $[x_1,\overline{\eta_1 \eta_2})$, therefore $d_X([x,\xi),F) \leq \varepsilon$. \qedhere
\end{proof}
\begin{cor}\label{cor two flats strongly asymptot}
    Let $X$ be a complete CAT(0) space, $F$ and $F'$ two flats in $X$, $x \in F$, and $\eta \in \bound F \cap \bound F'$. If $[x,\eta)$ is strongly asymptotic to $F'$, then for any interior point $\xi$ of $\bound F \cap \bound F'$, $[x,\xi)$ is strongly asymptotic to $F'$.
\end{cor}
\begin{proof}
    If $F$ or $F'$ is one-dimensional then the interior of $\bound F \cap \bound F'$ is empty. If not, let $\xi$ be an interior point. If $\angle(\eta,\xi)=\pi$, there exists a geodesic $\gamma \subset F$ containing $x$ and joining $\xi$ and $\eta$. $[x,\eta)$ is strongly asymptotic to $F'$ implies, by the flat stip theorem, that $\gamma \subset F'$. If $\angle(\eta,\xi)<\pi$ then, by convexity of $\bound F \cap \bound F'$, $\xi$ lies in a segment $\overline{\eta \eta'}$ for some $\eta' \in \bound F \cap \bound F'$ such that $\angle(\eta_1,\eta_2) < \pi$. The conclusion follows from \cref{strong asymptot is convex}. \qedhere
\end{proof}
\begin{lem}\label{asymptot to some geod}
    Let $X$ be a complete CAT(0) space, $F$ a flat in $X$, $\eta \in \bound F$, and $x \in X$ such that $[x,\eta)$ is strongly asymptotic to $F$. Then $[x,\eta)$ is strongly asymptotic to some geodesic in $F$.
\end{lem}
\begin{proof}
Let $b$ be a Busemann function with respect to $\eta$, and $H$ the intersection of $F$ with a level set of $b$, which is a hyperplane in $F$. Let $(f_n)_{n \in \mathbb{N}}$ be a sequence of points in $F$ such that $d_X(f_n,[x,\eta))\leq \frac{1}{n}$, and $h_n$ be the projection of $f_n$ on $H$ (i.e.\ the intersection of $H$ with the geodesic containing $f_n$ and $\eta$). By convexity of the distance function,
$$ d_X([x,\eta),[h_n,\eta)) = d_X([x,\eta),[f_n,\eta)) \leq d_X([x,\eta),f_n) \leq  \frac{1}{n}.$$
Let $n,m \in \mathbb{N}$, and $z\in [x,\eta)$ such that $d_X(z,[f_n,\eta)) \leq \frac{2}{n}$ and $d_X(z,[f_m,\eta)) \leq \frac{2}{m}$. So for any $a \in [f_n,\eta)$ and $b \in [f_m,\eta)$,
$$d_X([f_n,\eta)),[f_m,\eta))) \leq d_X(a,z)+d_X(z,b)$$
Therefore,
$$d_X([f_n,\eta)),[f_m,\eta))) \leq d_X(z,[f_n,\eta))+d_X(z,[f_m,\eta))\leq \frac{2}{n}+ \frac{2}{m}$$
Since $b(h_n) = b(h_m)$,
\begin{equation*}
    \begin{split}
        d_X(h_n,h_m) &= d_X([h_n,\eta),[h_m,\eta)) 
        \\& = d_X([f_n,\eta),[f_m,\eta))
        \\& \leq \frac{2}{n}+ \frac{2}{m}.
    \end{split}
\end{equation*}
So $(h_n)_n$ is a Cauchy sequence in $H$, and there exists $h \in H$ such that $h_n \xrightarrow[+\infty] \, h $, and a same argument as before shows that $d_X([h,\eta),[x,\eta)) =0$. So $[x,\eta)$ is strongly asymptotic to the extension, in $F$, of the geodesic ray $[h,\eta)$.\qedhere
\end{proof}
\noindent
For the rest of this section, $X$ is either a Euclidean building or a symmetric space.
\begin{prop}\label{all flat strongly asymptotic}
    Let $F,F'$ be singular flats such that $\dim \bound F = \dim \bound F' = \dim (\bound F \cap \bound F')$, and let $\eta$ be an interior point of a top-dimensional cell $c$ of $\bound F \cap \bound F'$. 
    \begin{itemize}
        \item[($1$)] If there exists $x \in F$ such that $[x,\eta)$ is strongly asymptotic to $F'$, then for any $y \in F$, $[y,\eta)$ is strongly asymptotic to $F'$.
        \item[($2$)] Let $s$ be a singular sphere in $\bound X$ containing $c$ and $\dim s = \dim \bound F = \dim \bound F'$, and let $\pi : X \to CS(s)$ be the projection via $\eta$. Then for any $x,y \in F$, $\pi(x) = \pi(y)$. Moreover, if there exist $ x \in F$ and $x' \in F'$ such that $\pi(x) = \pi(x')$, then $[x,\eta)$ is strongly asymptotic to $F'$ (and $[x',\eta)$ is strongly asymptotic to $F$).
    \end{itemize}
\end{prop}
\begin{proof}
    $(1)$ Let $\varepsilon>0$. There exists $a \in [x,\eta)$ such that $d_X(a,F') \leq \varepsilon$. Let $z' \in F'$ such that $d_X(a,z') \leq \varepsilon$. Let us denote by $c^{\mathrm{o}}$ the interior of $c$.
    For any $ \xi \in c^{\mathrm{o}}$, and any $t \geq 0$, by convexity of the distance function
    $$  d_X([a,\eta)(t),[z',\eta)(t)) \leq d_X(a,z') \leq \varepsilon.$$
    So $[a,c^{\mathrm{o}}) \subset N_{\varepsilon} (F')$. Since $c$ is a top-dimensional cell in $\bound F$ and $\eta \in c^{\mathrm{o}} $, for any $y \in F$, $[y,\eta)$ enters eventually in $[a,c^{\mathrm{o}})$. So
    $$ \lim_{t \to + \infty} d_X([y,\eta)(t),F') \leq \varepsilon.$$
    This holds for any $\varepsilon>0$, so $[y,\eta)$ is strongly asymptotic to $F'$. 

    \medskip \noindent
    $(2)$ Let $H$ be the singular flat with boundary $s$ to which $[x,\eta)$ and $[x',\eta)$ are strongly asymptotic. By \cref{asymptot to some geod}, there exist $z,z' \in H$ such that $[x,\eta)$ is strongly asymptotic to $[z,\eta)$ and $[x',\eta)$ is strongly asymptotic to $[z',\eta)$. In particular, $[z',\eta)$ is strongly asymptotic to $[x',\eta)$ and thus to $F'$. By $(1)$, $[z,\eta)$ is also strongly asymptotic to $F'$ since $z$ and $z'$ are both in $H$. $[x,\eta)$ is strongly asymptotic to $[z,\eta)$, so $[x,\eta)$ is strongly asymptotic to $F'$. The same argument show that $[x',\eta)$ is strongly asymptotic to $F$. \qedhere
\end{proof}
\begin{rem}
    The second point of \cref{all flat strongly asymptotic} implies that under these conditions, if there exist $x \in F$ and $x' \in F'$ such that $\pi(x) = \pi(x')$, then $\pi(F) = \pi(F')$.
\end{rem}
\begin{cor}\label{cor of prop}
    Let $F,F'$ be singular flats such that $\dim \bound F = \dim \bound F' = \dim (\bound F \cap \bound F')$, and let $\eta$ be an interior point of a top-dimensional cell $c$ of $\bound F \cap \bound F'$. 
    \begin{itemize}
        \item[(1)] For any $x,y \in F$, 
        $$d_X([x,\eta),F') = d_X([y,\eta),F').$$
        \item[(2)] Let $s$ be a singular sphere in $\bound X$ containing $c$ and $\dim s = \dim \bound F = \dim \bound F'$, and let $\pi : X \to CS(s)$ be the projection via $\eta$.  For any $ x \in F$ and $x' \in F'$, 
        $$ d(\pi(F),\pi(F')) = d(\pi(x),\pi(x')) = d_X([x,\eta),F') = d_X([x',\eta),F).$$
    \end{itemize}
\end{cor}
\begin{proof}
    $(1)$ If we denote by $s' = \bound F'$, by \cref{uniqueness projection} there exists a flat $F''$ with boundary $s'$, i.e.\ parallel to $F'$, to which $[x,\eta)$ is strongly asymptotic. By \cref{all flat strongly asymptotic}, $[y,\eta)$ is also strongly asymptotic to $F''$. Therefore,
    $$d_X([x,\eta),F') = d_X(F',F'') = d_X([y,\eta),F').$$
    $(2)$ Let $x\in F$ and $x' \in F'$. By \cref{all flat strongly asymptotic}, $\pi$ is constant on $F$ and on $F'$, so $d(\pi(F),\pi(F')) = d(\pi(x),\pi(x'))$, that we will denote by $D$. This implies that $d_X(F,F') \geq D$ ($\pi$ is $1$-Lipschitz) and in particular $d_X([x,\eta),F') \geq D$. Let $H$ (resp.\ $H'$) be the flat with boundary $s$ to which $[x,\eta)$ (resp.\ $[x',\eta)$) is strongly asymptotic. By \cref{asymptot to some geod}, there exists a geodesic $\gamma$ in $H$ to which $[x,\eta)$ is strongly asymptotic. $\gamma$ can be parameterized such that $d_X([x,\eta)(t),\gamma(t))\xrightarrow[+\infty] \, 0$. $H$ and $H'$ are parallel, so there exists a geodesic $\gamma'$ in $H'$, parallel to $\gamma$ such that for any $t$, $d_X(\gamma(t),\gamma'(t)) = d_X(H,H') = D $. Moreover, $\pi(F') = H'$ and they have the same dimension, so there exists $x'' \in F'$ such that $[x'',\eta)$ is strongly asymptotic to $\gamma'$ (by \cref{asymptot to some geod}). So for any $t\geq 0$,
    $$d_X([x,\eta)(t),[x',\eta)(t)) \leq d_X([x,\eta)(t),\gamma(t)) + d_X(\gamma(t),\gamma'(t)) + d_X(\gamma(t),[x',\eta)(t)).$$
    Therefore $d_X([x,\eta),[x',\eta)) \leq D$, and in particular $d_X([x,\eta),F') \leq D$.
    \qedhere
\end{proof}
\subsection{Generalized Iwasawa decomposition}\label{section Iwasawa}
Let us recall briefly the generalized Iwasawa decomposition for symmetric spaces, and we refer to \cite{eberlein1996geometry} for more details. Let $X$ be a symmetric space of non-compact type, $G = \textup{Isom}_0(X)$, $\mathfrak{g}$ its Lie algebra, $x_0 \in X$ a basepoint, $K = \textup{Stab}_G(x_0)$, $\mathfrak{g} = \mathfrak{k} \oplus \mathfrak{p}$ the Cartan decomposition with respect to $x_0$. Note that for any geodesic $\gamma$ through $x_0$, there exists $Y \in \mathfrak{p}$ such that for any $t$, $\gamma(t) = \exp(tY).x_0$. Let us fix $\gamma$ and $Y$, and consider $\mathfrak{a}$ a Cartan subspace of $\mathfrak{p}$ that contains $Y$, $\Phi$ the restricted root system of $\mathfrak{g}$ relative to $\mathfrak{a}$, and $\mathfrak{g} = \mathfrak{g}_0 \oplus \bigoplus_{\alpha \in \Phi} \mathfrak{g}_\alpha $ the restricted root space decomposition. 

\medskip \noindent
Let $\eta$ be the point at $+ \infty$ of $\gamma$, and let us denote $\mathfrak{a}_Y = \mathfrak{z}(Y)\cap \mathfrak{p}$, where $\mathfrak{z}(Y)$ is the centralizer of $Y$ in $\mathfrak{g}$, $\mathfrak{n}_Y = \bigoplus_{\alpha \in \Phi, \alpha(Y)>0} \mathfrak{g}_{\alpha}$, $A_\eta = \exp(\mathfrak{a}_Y)$, and $N_\eta = \exp(\mathfrak{n}_Y)$.
\begin{thm}[Generalized Iwasawa decomposition]\label{iwasawa}
    $$G = K A_\eta N_\eta.$$
    Moreover, we have
    \begin{itemize}
        \item[$(i)$] $A_\eta = A_\eta^{-1}$, and it normalizes $N_\eta$. 
        \item[$(ii)$] $A_\eta N_\eta$ acts simply transitively on $X$.
        \item[$(iii)$] For any $a \in A_\eta$, $a \gamma$ and $\gamma$ are parallel, i.e. $d_X(a.\gamma(t),\gamma(t))$ is constant.
        \\For any $n \in N_\eta$, $d_X(n.\gamma(t),\gamma(t)) \xrightarrow[+\infty] \, 0$.
    \end{itemize}
\end{thm}
\noindent
We refer to \cite[Chap.\ 2.19]{eberlein1996geometry}. Note that $A_\eta$ is not necessarily a subgroup: $\mathfrak{a}_Y$ is not necessarily a Lie subalgebra, unless $\eta$, i.e.\ $Y$, is regular. In this case $\mathfrak{a}_Y = \mathfrak{a}$, and we recover the usual Iwasawa decomposition. 
\section{Maximally distributed vertices in a spherical Coxeter complex}\label{section max distrib vertices}
Let $S$ be a spherical Coxeter complex, and $A \subset S$ a subset. If $A$ has diameter $<\pi$, we denote by $\hull(A)$ its convex hull in $S$. 
A hemisphere $\sigma$ in $S$ is called \textit{singular} if its boundary sphere $\partial \sigma$ is a wall.
\\Let $s$ be a singular sphere, and $\{\xi_i\}_{i}$ vertices in $S$. We say that the vertices $\{\xi_i\}_{i}$ span $s$ if $s$ is the smallest sphere (with respect to inclusion) that contains them.
\begin{prop}\label{choice}
Let $S$ be a spherical Coxeter complex of dimension $(n-1)$.
    \begin{itemize}
        \item[(1)] There exist vertices $\xi_1,\dots, \xi_n$ in $S$:
        \begin{itemize}
            \item[$(i)$] which are not pairwise opposite, nor all contained in a wall;
            \item[$(ii)$] for any $i = 1 \dots n$, $\{\xi_j\}_{j \ne i}$ span a wall;
            \item[$(iii)$] if $\sigma$ is a singular hemisphere in $S$ containing them, then $(n-1)$ of them must lie in its boundary wall $\partial \sigma$.
        \end{itemize}
        \item[(2)] For $i = 1, \dots, n$, let $s_i$ be the wall spanned by $\{\xi_j\}_{j \ne i}$ in $S$. Then $\displaystyle\bigcap_{i=1}^n s_i = \emptyset$. 
    \end{itemize}
\end{prop}
\begin{proof}
    $(1)$ The set of vertices satisfying $(i)$ and $(ii)$ is nonempty since it contains the vertices of a chamber. Take vertices $\xi_1,\dots, \xi_n$ satisfying $(i)$ and $(ii)$ such that $\hull(\xi_1,\dots, \xi_n)$ contains the maximum number of chambers, and let us show that they satisfy $(iii)$. Let $\sigma$ be a singular hemisphere containing all these vertices, and suppose that the interior of $\sigma$ contains more than one vertex. To simplify notations, suppose that $\xi_1,\dots,\xi_p \in \sigma \backslash \partial \sigma $, with $p\geq 2$. We use $\xi_1$ to ``push'' $\xi_2, \dots, \xi_p$ to $\partial \sigma$. For $i = 2 \dots p$, extend the geodesic segment $\overline{\xi_1\xi_i}$ to $\overline{\xi_1\xi_i'}$, where $\xi_i'$ is its first intersection with $\partial \sigma$. It is clear that the convex hull of $\xi_1, \xi_2', \dots, \xi_p',\xi_{p+1}, \dots, \xi_n$ is bigger than the initial one, and we claim that they still satisfy $(i)$ and $(ii)$, which leads to a contradiction. 

    \medskip \noindent
    Indeed, $\xi_1$ is clearly not opposite to any $\xi_i'$. If some $\xi_i'$ is opposite to some $\xi_j'$ (similarly if some $\xi_i'$ is opposite to some $\xi_j$), then $\xi_1,\xi_i,\xi_j$ lie in a same singular $1$-sphere, therefore any $(n-1)$ vertices of $\xi_1,\dots, \xi_n$ that contain them span a sphere of dimension $<n-2$, which contradicts the fact that $\xi_1,\dots, \xi_n$ satisfy $(ii)$. If $\xi_1, \xi_2', \dots, \xi_p',\xi_{p+1}, \dots, \xi_n$ are contained in a wall $s$, then their convex hull is in $s$, in particular $\xi_1,\dots, \xi_n \in s$, which contradicts $(i)$. Therefore they still satisfy $(i)$. 

    \medskip \noindent
    Finally, note that $\xi_2', \dots, \xi_p',\xi_{p+1}, \dots, \xi_n$ are all in $\partial \sigma$, which is a wall, and they do not span a smaller singular sphere, otherwise the convex hull would not be $(n-1)$-dimensional. Also, for any $k \in \{2, \dots, p\}$, $\xi_1, \xi_2', \dots,\xi'_{k-1},\xi'_{k+1}, \dots, \xi_p',\xi_{p+1}, \dots, \xi_n$ are in the same wall that $\xi_1, \dots,\xi_{k-1},\xi_{k+1}, \dots, \xi_n$ span, and they cannot span a smaller sphere by the same argument. Similarly if one removes one of $\xi_{p+1}, \dots, \xi_n$. Therefore $\xi_1, \xi_2', \dots, \xi_p',\xi_{p+1}, \dots, \xi_n$ also satisfy $(ii)$. 
    
    \medskip \noindent
    $(2)$ Let us denote the opposites of $\xi_1,\dots, \xi_n$ in $S$ by $\hat{\xi_1},\dots, \hat{\xi_n}$.
    Note that, by the same argument as before, for any $\xi_{k_1},\dots, \xi_{k_p}$, the singular sphere that they span is $(p-1)$-dimensional. In particular, for any $i\in \{1,\dots,n\}$, $\xi_i, \hat{\xi_i} \not \in s_i$. Moreover, $s_i\cap s_j$ is spanned by $\{\xi_k\}_{k \ne i,j}$. So $\cap_{i \ne j} s_i = \{ \xi_j, \hat{\xi_j} \}$, which are not in $s_j$. \qedhere
\end{proof}
\begin{rem}
If the spherical Coxeter complex $S$ is an apartment in $\bound X$, then the choice of the vertices $\xi_1,\dots, \xi_n$ no longer depends on $S$: if $S'$ is another apartment containing them, then they also satisfy $(1)$ and $(2)$ of \cref{choice} in $S'$.
\end{rem}
\begin{defn}\label{def max distributed vertices}
    Let $S$ be a spherical Coxeter complex of dimension $(n-1)$, and $\xi_1,\dots, \xi_n \in S$ vertices. We say that $\xi_1,\dots, \xi_n$ are \emph{maximally distributed} if they satisfy the condition $(1)$ in Proposition \ref{choice}. See \cref{fig max distributed vertices} for some examples.
\end{defn}

\medskip\noindent
\begin{figure}[!ht]
    \centering
    \def\svgwidth{0.27\textwidth}
    {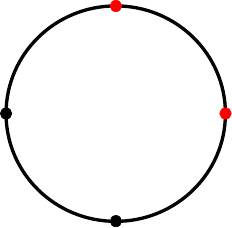}
    \, \, \, \, \, 
    \def\svgwidth{0.27\textwidth}
    {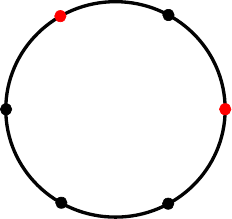}
    \, \, \, \, \, 
    \def\svgwidth{0.30\textwidth}
    {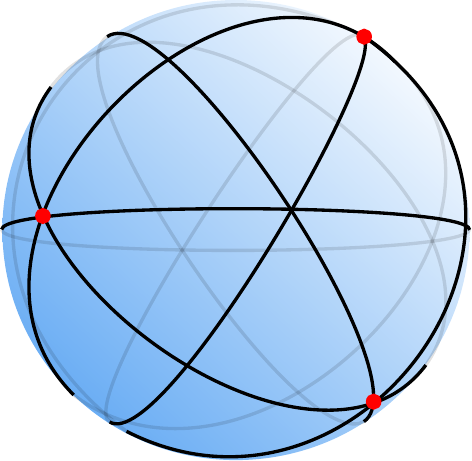}
    \caption{Maximally distributed vertices in $A_1 \times A_1$, $A_2$, and $A_3$.}
    \label{fig max distributed vertices}
\end{figure}

\medskip\noindent
\begin{prop}\label{fibers}
    Let $X$ be a symmetric space or a Euclidean building, $F$ a maximal-flat in $X$, and $\xi_1,\dots, \xi_n$ maximally distributed vertices in $\bound F$. For $i=1, \dots, n$, let $s_i$ be the wall in $\bound F$ spanned by $\{\xi_j\}_{j \ne i}$, and let $\eta_i$ be an interior point of a top-dimensional cell of $\hull(\{\xi_j\}_{j \ne i}) \subset s_i$ (see \cref{fig max distrib A3} for an example in $A_3$). Let $\pi_i$ be the projection map onto $CS(s_i)$ via $\eta_i$. Then for any maximal-flat $F'$ containing all the $\xi_i$'s at infinity, the fibers of $\pi_i$ in $F'$ are the $(n-1)$-flats with $\{\xi_j\}_{j \ne i}$ at infinity. In particular, $\pi_i$ is constant on the $(n-1)$-flat containing $\{\xi_j\}_{j \ne i}$ at infinity.
\end{prop}
\begin{figure}[!ht]
    \centering
    \def\svgwidth{0.4\textwidth}
    {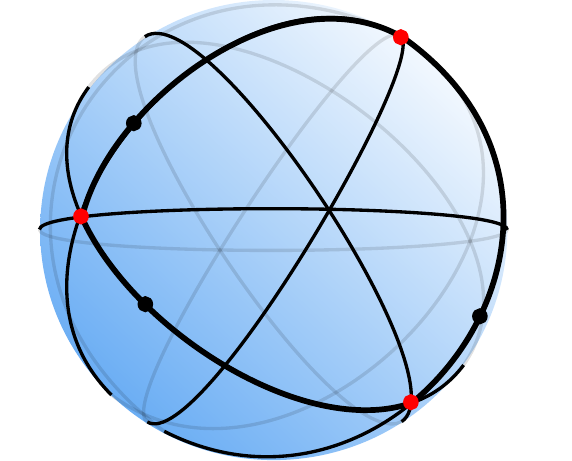}
    \caption{}
    \label{fig max distrib A3}
\end{figure}
\begin{proof}
    Suppose $i=1$, let $z\in F'$, and let $H'$ be the $(n-1)$-flat in $F'$ containing $z$ and containing $\xi_2,\dots \xi_n$ at infinity. Let us denote $H = \pi_1(z)$. $\eta_1$ is an interior point of a top-dimensional cell of $\bound H \cap \bound H'$, and $[z,\eta_1)$ is strongly asymptotic to $H$, so for any $y \in H'$ $[y,\eta_1)$ is also strongly asymptotic to $H$ by \cref{all flat strongly asymptotic}, i.e.\ $\pi_1(H') = \pi_1(z)$. 
    \\ If $z' \in F'$ is not contained in $H'$, then $\pi_1(z') \ne \pi_1(z)$ by \cref{cor of prop}.\qedhere 
\end{proof}
\section{Proof of the main result}\label{section proof main result}
Let us recall the setting. Let $X$ be a Euclidean building or a symmetric space of non-compact type of rank $n$. Let $F_0$ be a maximal-flat, and $\xi_1,\dots, \xi_n \in \bound F_0$ maximally distributed vertices, see \cref{def max distributed vertices}. For all $i=1,\dots, n$, let $s_i$ be the wall in $\bound F_0$ spanned by $\{\xi_j\}_{j \ne i}$, and $\eta_i$ an interior point of a top-dimensional cell of $\hull(\{\xi_j\}_{j \ne i}) \subset s_i$. 

\medskip\noindent
We denote $\Delta = \hull(\xi_1,\dots, \xi_n)$, and let $X_{\Delta}$ be the union of maximal-flats in $X$ that contain $\Delta$ at infinity, and let $\pi_i$ be the projection onto $CS(s_i)$ via $\eta_i$ defined in Definition \ref{proj}. Finally, let 
$$\pi : X_{\Delta} \to CS(s_1) \times \dots \times CS(s_n)$$ 
be the restriction of the product map $\pi_1 \times \dots \times \pi_n$ to $X_{\Delta}$. Each $CS(s_i)$ is equipped with the Hausdorff distance $d_i$, and we equip the product space with the $L^1$ product metric. We will show the following, which is a restatement of \cref{mainthm general}.
\begin{thm}\label{mainthm general2}
\noindent
    \begin{itemize}
        \item[(1)] If $X$ is a Euclidean building, $\pi$ is a bi-Lipschitz map. Moreover, the inclusion of $X_{\Delta}$, equipped with the path-metric, in $X$ is a bi-Lipschitz embedding, therefore it induces a bi-Lipschitz embedding from $CS(s_1) \times \dots \times CS(s_n)$ into $X$.
        \item[(2)] If $X$ is a symmetric space of non-compact type, $\pi$ is a quasi-isometry. Moreover, the inclusion of some $\delta$-neighborhood of $X_{\Delta}$, equipped with the path-metric, in $X$ is a quasi-isometric embedding, therefore it induces a quasi-isometric embedding from $ CS(s_1) \times \dots \times CS(s_n)$ into $X$.
    \end{itemize}
\end{thm}
\medskip\noindent
We refer to \cref{subsect overview} for an overview of the proof, and its steps.

\medskip\noindent
Throughout the proof, $n$ is the rank of $X$. Since \cref{mainthm general} is trivial in rank $1$, we suppose $n\geq 2$. In particular, the angular $\angle$ and Tits metric $d_T$ coincide in $\bound X$.

\medskip \noindent
\textbf{Step $1$: $\pi$ is a bi-Lipschitz embedding when restricted to a flat containing $\Delta$ at infinity (with uniform constants):}
\medskip \noindent
\\Let $F \subset X_{\Delta}$ be a maximal-flat in $X$ containing $\Delta$ at infinity.

\medskip \noindent
• \textit{Substep 1: $\pi$ is injective.} 

\medskip \noindent
Let $x,y \in F$ such that for all $i= 1 , \dots , n$, $\pi_i(x) = \pi_i(y)$. By \cref{fibers}, $\pi_i(x) = \pi_i(y)$ implies that $x$ and $y$ are in a same $(n-1)$-flat $H_i \subset F$ such that $\{\xi_j\}_{j \ne i} \subset \bound H_i$. The flats $H_i$ intersect in a single point because $\cap_{i=1}^n \bound H_i = \emptyset$ by \cref{choice}. Therefore $x=y$.

\medskip \noindent
• \textit{Substep 2: $\pi$ is a bi-Lipschitz embedding.} 

\medskip \noindent
The goal is to show that there exists $\alpha >0 $ such that for any $x,y \in F$ $d_X(x,y) \leq \alpha \sum_{i=1}^n d_i(\pi_i(x),\pi_i(y)) $.

\medskip \noindent
Let us denote the opposites of $\xi_1,\dots, \xi_n$ in $\bound F$ by $\hat{\xi_1},\dots, \hat{\xi_n}$. Let us start by noting that for all $i=1, \dots, n$, $\pi_i(F)$ is a constant speed path in $CS(s_i)$. Also, by \cref{fibers}, when moving in $F$ along a geodesic going to $\xi_i$, only $\pi_i$ changes, the rest of the projections are constant. Therefore, to go from $x$ to $y$, we start from $x$ by following a geodesic with endpoints $\{ \xi_1, \hat{\xi_1} \}$ until we equalize $\pi_1$, then we do the same for the other directions. After equalizing all $\pi_i$'s, by injectivity in the previous substep, we would have reached $y$.
\medskip \noindent
\\ Let $\theta_i = d_T(\xi_i, s_i) $, and let $x,y$ be in a same geodesic with endpoints $\xi_i,\hat{\xi_i}$. If $\theta_i = \pi/2$, $d_X(x,y) = d_i(\pi_i(x),\pi_i(y))$. \\If $\theta_i < \pi/2$, $d_X(x,y) = \alpha_i \, d_i(\pi_i(x),\pi_i(y))$, where $\alpha_i = \frac{1}{\tan(\theta_i)}$. Let us denote $\alpha = \displaystyle\max_{i} \alpha_i$. By concatenating such paths, we have $\forall x,y \in F$, $ d_X(x,y) \leq \alpha \sum_{i=1}^n d_i(\pi_i(x),\pi_i(y)) $. Note that $\alpha $ is independent of the maximal-flat $F$, and the previous inequality holds whenever $x$ and $y$ are in a same maximal-flat containing $\Delta$ at infinity.
\begin{rem}
    Let us note that the maximal distribution of the vertices $\xi_1,\dots, \xi_n$ is not needed for the injectivity of $\pi$ inside a same flat containing $\Delta$ at infinity. We only needed that the walls spanned by $\{\xi_j\}_{j \ne i}$, for any $i$, have trivial intersection. So the result still holds in a same flat if $\Delta$ consisted of a single chamber for example. However, the injectivity fails without the the maximal distribution of the vertices if one considers different maximal flats with $\Delta$ at infinity.
\end{rem}

\medskip \noindent
\textbf{Step $2$: $\pi$ is a bi-Lipschitz (resp.\ quasi-isometric) embedding:}

\medskip \noindent
• \textit{Substep 1: $\pi$ is injective (resp.\ quasi-injective).} 

\medskip \noindent
By quasi-injective, we mean that for any $\delta$ big enough, there exists $D\geq 0$ such that for any $x,y \in X_{\Delta}$, if $d_i(\pi_i(x),\pi_i(y)) \leq \delta$ for all $i=1,\dots,n$ then $d_X(x,y) \leq D$. 
\\The proof for Euclidean buildings and for symmetric spaces is similar, but, for the sake of clarity, we will treat them separately. 

\medskip \noindent
\underline{$X$ is a Euclidean building:}
Let $x,y \in X_{\Delta}$ such that $\pi_i(x) = \pi_i(y)$ for $i=1,\dots,n$. If they lie in a same maximal-flat containing $\Delta$ at infinity, then we are done by the previous step. If not, let $F_x$ and $F_y$ be maximal-flats such that $x \in F_x$, $y \in F_y$ and $\Delta \subset \bound F_x \cap \bound F_y$. $\Delta \subset \bound F_x \cap \bound F_y$, so $F_x$ and $F_y$ share a chamber at infinity and must intersect \cite[Lemma 4.6.5]{kleiner1997rigidity}. Moreover, by \cite[Cor.\ 4.4.6]{kleiner1997rigidity}, their intersection is a Weyl polyhedron, i.e.\ an intersection of singular half-spaces $\{M_i\}_{i\in I}$ of $F_x$, so $F_x \cap F_y = \cap_{i\in I} M_i$. Note that for all $i\in I$, $\Delta \subset \bound M_i$. By assumption, $x \not \in F_y$, so it is not inside one of these singular half-spaces. Let us denote it by $M$, and let $H$ be its boundary wall. For every $i=1,\dots,n$, $\xi_i \in \bound M$, and $\bound M$ is a singular hemisphere in $\bound F_x$. Since $\xi_1,\dots, \xi_n$ are maximally distributed, by \cref{choice}, $\exists i \in \{1, \dots,n\}$ such that for all $j \ne i$, $\xi_j \in \bound H$, and in particular $\eta_i \in \bound H$. With loss of generality, we suppose $i=1$. This implies that $[x,\eta_1)$ stays parallel to $H$ (because $\eta_1 \in \bound H$), and therefore never enters in $M$, see \cref{flat intersec}:
$$d_X([x,\eta_1),H) = d_X(x,H)>0.$$
\begin{figure}[!ht]
    \centering
    \def\svgwidth{0.45\textwidth}
    {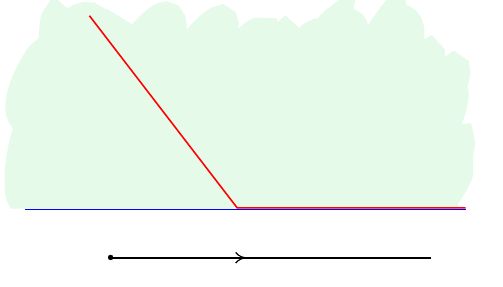}
    \caption{$[x,\eta_1)$ staying outside of $M$.}
    \label{flat intersec}
\end{figure}
On the other hand, $\pi_1(x) = \pi_1(y)$ implies, by \cref{all flat strongly asymptotic}, that $[x,\eta_1)$ is strongly asymptotic to $F_y$. However, $F_x \cap F_y \subset M$ and $H$ is convex so
$$ 0 = d_X([x,\eta_1),F_y) = d_X([x,\eta_1),F_x \cap F_y) \geq d_X([x,\eta_1),M) = d_X([x,\eta_1),H) >0.$$
We get a contradiction. So the condition $\pi_i(x) = \pi_i(y)$ for $i=1,\dots,n$ implies that $x$ and $y$ lie in same maximal-flat containing $\Delta$ at infinity, and we are done.

\medskip \noindent
\underline{$X$ is a symmetric space:} 
let $x,y \in X_{\Delta}$ such that $d_i(\pi_i(x),\pi_i(y)) \leq 1$ for any $i=1,\dots,n$. Let $F_x$ and $F_y$ be maximal-flats such that $x \in F_x$, $y \in F_y$, and $\Delta \subset \bound F_x \cap \bound F_y$. In particular, $F_x$ and $F_y$ share a chamber at infinity so $d_X(F_x,F_y) = 0$. The following lemma is a special case of \cite[Lemma B.1]{eskin1998quasi} (see also \cite[Chapter 7]{mostow1973strong}).
\begin{lem}
    There exist constants $\lambda_0$ and $\lambda$ depending only on $X$ such that the following holds: if $F_1$ and $F_2$ are maximal-flats in $X$ with $d_X(F_1,F_2) = 0$, then for any $\delta \geq \lambda_0$, there exist singular convex polyhedrons $P$ and $P'$ in $F_1$ (i.e.\ intersections of singular half-spaces in $F_1$) such that $d_{Haus}(P,P') \leq \lambda \delta$ and
    $$P' \subset F_1 \cap N_{\delta}(F_2) \subset P.$$
\end{lem}
\noindent
Let us fix $\delta > \lambda_0$, and let $P$ and $P'$ be the singular convex polyhedrons in $F_x$ such that $d_{Haus}(P,P') \leq \lambda \delta$, $P' \subset F_x \cap N_{\delta}(F_y) \subset P$, and let $(M_i)_{i\in I}$ be the singular half-spaces in $F_x$ such that $P = \cap_{i\in I} M_i$. Note that for any $i \in I$, $\Delta \subset \bound M_i$. Let us show that the condition $d_i(\pi_i(x),\pi_i(y)) \leq \delta$ for any $i=1,\dots,n$ implies that $x \in P$. If $x \in M_i$ for any $i \in I$, we are done. If not, let us denote this singular half-space by $M$, and let $H$ be its wall. Again, $\xi_1,\dots, \xi_n \in \bound M$ are maximally distributed, so $\exists i \in \{1, \dots,n\}$ such that for all $j \ne i$, $\xi_j \in \bound H$, and in particular $\eta_i \in \bound H$. We assume again that $i=1$. So $[x,\eta_1)$ stays parallel to $H$: $d_X(x,H) = d_X([x,\eta_1),H)$.
\\On the other hand, by \cref{cor of prop}, $d_1(\pi_1(x),\pi_1(y)) \leq \delta$ implies that $d_X([x,\eta_1),F_y) \leq \delta$. Since $[x,\eta_1)\subset F_x$, we get
$$   d\left([x,\eta_1),F_x \cap N_{\delta}(F_y)\right) = 0 .      $$
$F_x \cap N_{\delta}(F_y) \subset P \subset M$, so $d\left([x,\eta_1),M\right) = 0$. However, $x \not\in M$ and $H$ is the boundary wall of $M$, so
$$ d\left([x,\eta_1),M\right) = d\left([x,\eta_1),H\right) = d_X(x,H).$$
Therefore $x \in H$, and we get a contradiction. 

\medskip \noindent
Hence, the condition $d_i(\pi_i(x),\pi_i(y)) \leq \delta$ for all $i=1,\dots,n$ implies that $x \in P$. Therefore $x \in N_{\lambda}(F_x \cap N_{\lambda \delta}(F_y))$, and in particular, $x \in N_{\lambda+\lambda\delta}(F_y)$. Let $y' \in F_y$ such that $d_X(x,y')\leq \lambda+\lambda \delta$. The projections are $1$-Lipschitz, so for any $i$, $d_i(\pi_i(x),\pi_i(y'))\leq \lambda+\lambda\delta$. $y$ and $y'$ lie in a same flat, so by step $1$
$$d_X(y,y') \leq \alpha \sum_{i=1}^n d_i(\pi_i(y),\pi_i(y')). $$
Therefore
\begin{equation}\label{quasi-inj in ss}
    \begin{split}
        d_X(x,y) &\leq d_X(x,y')+d_X(y',y)
        \\ &\leq \lambda+\lambda\delta +  \alpha \sum_{i=1}^n d_i(\pi_i(y),\pi_i(y'))
        \\ &\leq \lambda+\lambda\delta +  \alpha \left( \sum_{i=1}^n d_i(\pi_i(y),\pi_i(x)) + \sum_{i=1}^n d_i(\pi_i(x),\pi_i(y'))  \right)
        \\ &\leq \lambda+\lambda\delta +  \alpha \left( 0 + n (\lambda+\lambda\delta)  \right)
        \\&\leq (1+\alpha n)(\lambda+\lambda\delta).
    \end{split}
\end{equation}
Note that $D = (1+\alpha n)(\lambda+\lambda\delta)$ depends only on $X$.

\medskip \noindent
• \textit{Substep 2: $\pi$ is a bi-Lipschitz (resp.\ quasi-isometric) embedding}

\medskip \noindent
Again, for the sake of clarity, let us treat the Euclidean buildings and symmetric spaces separately.

\medskip \noindent
\underline{$X$ is a Euclidean building:}
let $x,y \in X_{\Delta}$. If they lie in a same maximal-flat containing $\Delta$ at infinity, we are done by step $1$. If not, let $F_x$ and $F_y$ be maximal-flats such that $x \in F_x$, $y \in F_y$ and $\Delta \subset \bound F_x \cap \bound F_y$. $\pi_i([x,\xi_i))$ and $\pi_i([y,\xi_i))$ are not necessarily equal, but they still share a ray. Indeed, since $CS(s_i)$ is a Euclidean building of dimension $1$, i.e.\ a metric tree, $[x,\xi_i)$ and $[y,\xi_i)$ stay at bounded distance near $+ \infty$, and $\pi_i$ is 1-Lipschitz, then $\pi_i([x,\xi_i))$ and $\pi_i([y,\xi_i))$ share a ray. We have two cases, either one of $\pi_i([x,\xi_i))$, $\pi_i([y,\xi_i))$ is a subset of the other. Or there is a branching and they form a tripod. Let us start from $i=1$.

\medskip \noindent
If we are in the first case, i.e.\ if $\pi_1([y,\xi_1)) \subset \pi_1([x,\xi_1))$, we can start from $x$ and follow $[x,\xi_1)$ until we equalize $\pi_1$. In other words, $\exists x_1 \in [x,\xi_1)$ such that $\pi_1(x_1) = \pi_1(y)$. Since $ x_1 \in [x,\xi_1)$, $\forall i \ne 1$, $\pi_i(x_1) = \pi_i(x)$. And $x,x_1 \in [x,\xi_1) \subset F_x$ so $d_X(x,x_1)\leq \alpha d_1(\pi_1(x), \pi_1(x_1)) = \alpha d_1(\pi_1(x), \pi_1(y)) $. We denote by $y_1 = y$, and move to $i=2$ to equalize $\pi_2$. If $\pi_1([x,\xi_1)) \subset \pi_1([y,\xi_1))$, we take $y_1 \in [y,\xi_1) $ such that $\pi_1(y_1) = \pi_1(x)$, and denote $x_1 = x$.

\medskip \noindent
If we are in the second case and there was a branching, let $x_1 \in [x,\xi_1)$ and $y_1 \in [y,\xi_1)$ such that $\pi_1(x_1) = \pi_1(y_1)$ is the branching point. Since $x,x_1 \in [x,\xi_1) \subset F_x$  and $y,y_1 \in [y,\xi_1) \subset F_y$, we have
\begin{equation}\label{+}
\begin{split}
     d_X(x,x_1)\leq \alpha \, d_1(\pi_1(x), \pi_1(x_1)),
    \\ d_X(y,y_1)\leq \alpha \, d_1(\pi_1(y), \pi_1(y_1)).
\end{split}
\end{equation}
$\pi_1(x_1) = \pi_1(y_1)$ is the branching point, so $d_1(\pi_1(x), \pi_1(y))= d_1(\pi_1(x), \pi_1(x_1))+d_1(\pi_1(y_1), \pi_1(y))$ and \ref{+} implies that 
\begin{equation}\label{sum}
    d_X(x,x_1)+d_X(y,y_1) \leq \alpha \, d_1(\pi_1(x), \pi_1(y)).
\end{equation}
We repeat this process by starting from $x_1, y_1$ and we equalize $\pi_2$. We get at the end a path $x = x_0 \to x_1 \to \dots \to x_{n-1} \to x_{n} = y_n \to y_{n-1} \to \dots \to y_1 \to y_0 = y$. Note that $x_n = y_n$ by injectivity because we've equalized all $\pi_i$. By the triangle inequality and by \ref{sum}:
$$ d_X(x,y) \leq \sum_{i=1}^n d_X(x_{i-1},x_i) + d_X(y_{i-1},y_i) \leq \sum_{i=1}^n \alpha \, d_i(\pi_i(x),\pi_i(y)) .$$

\medskip \noindent
\underline{$X$ is a symmetric space:}
Let $x,y \in X_{\Delta}$. If they lie in a same maximal-flat containing $\Delta$ at infinity, we are done by step $1$. If not, let $F_x$ and $F_y$ be maximal-flats such that $x \in F_x$, $y \in F_y$ and $\Delta \subset \bound F_x \cap \bound F_y$. $\pi_1([x,\xi_1))$ and $\pi_1([y,\xi_1))$ are again geodesic rays with the same point at $+ \infty$ in $CS(s_1)$, which is a rank one symmetric space of non-compact type. The difference with the building case is that the rays no longer share a ray. To overcome this, we need the following lemma.
\begin{lem}\label{lem convergence of geodesics}
    Let $X$ be a symmetric space of non-compact type. For any regular point $\eta \in \bound X$, there exists $\delta >0$ such that if $\textup{Isom}_0(X) = KAN$ is an Iwasawa decomposition with respect to $\eta$, and $x,y \in X$ are in a same $N$-orbit, then
    $$      d_X\left( [x,\eta)(d), [y,\eta)(d) \right) \leq \delta,$$
    where $d=d_X(x,y).$
\end{lem}
\noindent
For a proof, see for example the proof of \cite[Lemma 4]{leuzinger2000corank}. As an application, since in rank one the stabilizer of any point acts transitively on the boundary, $\delta$ does not depend on $\eta$. If $\eta \in \bound X$, we denote by $b_{\eta}$ a Busemann function with respect to $\eta$ and some base point in $X$. Note that if $x,y \in X$, $b_{\eta}(x) - b_{\eta}(y)$ does not depend on the basepoint. So, we have the following.
\begin{cor}[see \cref{quasi-path rank 1}]\label{cor quasi path in rank one}
    Let $X$ be a rank one symmetric space of non-compact type. There exists $\delta$ such that for any $\eta \in \bound X$ and for any $x,y \in X$ the following holds: 
    \\If $b_{\eta}(x)-b_{\eta}(y) \geq 0$, let $x' \in [x,\eta)$ such that $b_{\eta}(x') = b_{\eta}(y)$, let $x'' \in [x',\eta)$ such that $d_X(x',x'') = d_X(x',y)$, and $y'' \in [y,\eta)$ such that $d_X(y,y'') = d_X(x',y)$. Then the path
    $$x \xrightarrow[]{} x'' \xrightarrow[]{}y'' \xrightarrow[]{} y,$$
    where each arrow is a geodesic segment, has total length $\leq 3 \, d_X(x,y)+ \delta$.
\end{cor}
\begin{figure}[!ht]
    \centering
    \def\svgwidth{0.3\textwidth}
    {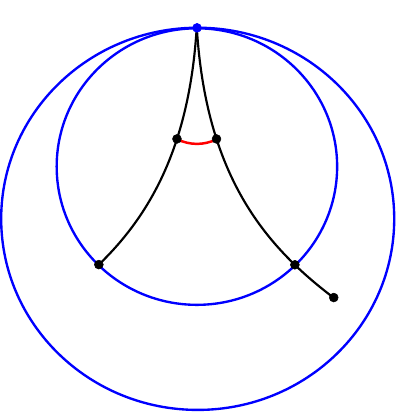}
    \caption{The quasi-geodesic $x \xrightarrow[]{} x'' \xrightarrow[]{}y'' \xrightarrow[]{} y$.}
    \label{quasi-path rank 1}
\end{figure}
\begin{proof}
    Let $B$ be the horoball whose boundary horosphere contains $y$, and let $p : X \to B$ be the projection. $p(x) = x'$ so
    \begin{equation}\label{eq 1 cor}
        d_X(x,x') \leq d_X(x,y).
    \end{equation}
    $B$ is convex so $p$ is $1$-Lipschitz \cite[Chap.\ 2 Prop.\ 2.4]{bridson2013metric}, and we have
    \begin{equation}\label{eq 2 cor}
        d_X(x',y) = d_X(p(x),p(y)) \leq d_X(x,y).
    \end{equation}
    $b_{\eta}(x') = b_{\eta}(y)$ and $X$ is rank one so $x'$ and $y$ are in the same $N$-orbit for some Iwasawa decomposition. By \cref{lem convergence of geodesics}, $d_X(x'',y'')\leq \delta$, and the path $x' \xrightarrow[]{} x'' \xrightarrow[]{}y'' \xrightarrow[]{} y$ has length $\leq 2 d_X(x',y) + \delta$. We conclude by using \ref{eq 1 cor} and \ref{eq 2 cor} in $d_X(x,y) \leq d_X(x,x')+d_X(x',y)$.\qedhere
\end{proof}
\noindent
Let us go back to the proof. Let $\delta>0$ be as in \cref{cor quasi path in rank one}, that works for all the cross sections $CS(s_i)$, for $i=1,\dots,n$. Let us start from $i=1$.
\medskip \noindent
\\Let $x_1 \in [x,\xi_1)$ and $y_1 \in [y,\xi_1)$ such that, as in \cref{cor quasi path in rank one}, the path, in $CS(s_1)$, $\pi_1(x) \xrightarrow[]{} \pi_1(x_1) \xrightarrow[]{} \pi_1(y_1) \xrightarrow[]{} \pi_1(y)$ has length $\leq 3 \, d_1(\pi_1(x),\pi_1(y)) + \delta$.
$$d_1(\pi_1(x),\pi_1(x_1)) + d_1(\pi_1(x_1),\pi_1(y_1)) + d_1(\pi_1(y_1),\pi_1(y)) \leq 3 \, d_1(\pi_1(x),\pi_1(y)) + \delta.$$
$x,x_1 \in [x,\xi_1) \subset F_x$, and $y,y_1 \in [y,\xi_1) \subset F_y$, so
\begin{equation*}
\begin{split}
     d_X(x,x_1)\leq \alpha \, d_1(\pi_1(x), \pi_1(x_1)),
    \\ d_X(y,y_1)\leq \alpha \, d_1(\pi_1(y), \pi_1(y_1)).
\end{split}
\end{equation*}
Therefore,
\begin{equation}\label{sum ss}
    \begin{split}
        d_X(x,x_1)+d_X(y,y_1) &\leq \alpha  \left(  d_1(\pi_1(x), \pi_1(x_1)) + d_1(\pi_1(y), \pi_1(y_1)) \right)
        \\ & \leq 3 \, \alpha \, d_1(\pi_1(x),\pi_1(y)) + \alpha \delta.
    \end{split}
\end{equation}
We repeat this process by starting from $x_1, y_1$. We get at the end a path $x = x_0 \to x_1 \to \dots \to x_{n-1} \to x_{n} \to y_n \to y_{n-1} \to \dots \to y_1 \to y_0 = y$. Let us note that, unline in the building case, $x_n \ne y_n$. However, since $\pi(x_n) = \pi(y_n)$ for all $i$, by the quasi-injectivity in the previous step,
$$d_X(x_n,y_n) \leq D,$$
where $D$ is the constant in $\ref{quasi-inj in ss}$. By the triangle inequality and by \ref{sum ss}:
\begin{equation*}
    \begin{split}
        d_X(x,y) &\leq \sum_{i=1}^n \big( d_X(x_{i-1},x_i) + d_X(y_{i-1},y_i) \big) + d_X(x_n,y_n) 
        \\ &\leq \sum_{i=1}^n \big( 3 \, \alpha \, d_i(\pi_i(x),\pi_i(y)) + \alpha \delta \big) + D 
         \\ &\leq 3 \alpha \,  \sum_{i=1}^n \big( d_i(\pi_i(x),\pi_i(y))\big) + (n \alpha \delta + D).
    \end{split}
\end{equation*}

\medskip \noindent
\textbf{Step $3$: $\pi$ is surjective (resp.\ quasi-surjective):}

\medskip \noindent
By quasi-surjective, we mean that if $X$ is a symmetric space and if $(H_1',\dots, H_n') \in CS(s_1)\times \dots \times CS(s_n)$, there exists $x \in X_\Delta$ such that for any $i= 1, \dots,n$, $d_i(\pi_i(x),H_i')) \leq 1$.

\medskip \noindent
Let us start by the following observation.
\begin{lem}\label{parallel sets are in X_Delta}
   For any $i=1,\dots,n$, $P(s_i) \subset X_\Delta$.
\end{lem}
\begin{proof}
    Let $i \in \{1, \dots,n\}$, and let $H$ be the flat in $F_0$ with boundary $s_i$. Let $\sigma$ be the singular hemisphere of $\bound F_0$ bounded by $s_i$ and containing $\xi_i$, and let $m$ be its center. Therefore, $m \in \partial_{\infty} CS(s_i)$. 
    \\Now let $H' \in CS(s_i)$. Then any geodesic in $CS(s_i)$ that contains it and contains $m$ at infinity corresponds to a maximal-flat $F$ in $X$ whose boundary contains $s_i$ and $m$. By convexity, $\bound F$ contains $\sigma$ and therefore contains $\xi_i$. Recall that for any $j \ne i$, $\xi_j \in s_i \subset \bound F$. We conclude that $\Delta \subset \bound F$.
\end{proof}
\medskip \noindent
Let $x \in X_\Delta$, and $F$ be the flat that contains $x$ such that $\Delta \subset \bound F$. Let us denote $\pi_i(x) = H_i$ for any $i$. Let $H_1' \in CS(s_1)$. 
\noindent
\\Since $\pi_1(x) = H_1$, there exists $x_1 \in H_1$ such that $[x,\eta_1)$ is strongly asymptotic to $[x_1,\eta_1)$. As in \cref{parallel sets are in X_Delta}, let $m_1$ be the center of the singular hemisphere $\sigma_1$ bounded by $s_1$ and containing $\xi_1$. By considering, in $CS(s_1)$, two geodesic rays containing $H_1$ and $m_1$ at infinity (resp.\ $H_1'$ and $m_1'$), there exist two flats $F_1$ and $F_1'$ such that $H_1 \subset F_1$, $H_1' \subset F_1'$, and $\sigma_1 \subset \bound F_1 \cap \bound F_1'$. Moreover, since for any regular point $\mu$ in $\sigma_1$, $[x,\mu)$ is strongly asymptotic to $F_1'$, it also holds, by \cref{cor two flats strongly asymptot}, for any interior point of $\sigma_1$, and in particular for $\xi_1$. By \cref{asymptot to some geod}, there exists $x_1' \in F_1'$ such that $[x_1,\xi_1)$ is strongly asymptotic to $[x_1',\xi_1)$. $H_1'$ is transverve to $[x_1',\xi_1)$, so $x_1'$ can be taken in $H_1'$. Let us denote the opposites of $\xi_1$ and $m_1$ in $F_1'$ by $\xi_1'$ and $m_1'$. see \cref{figure surj in building} for the building case.
\begin{figure}[!ht]
    \centering
    \def\svgwidth{420pt}{
    \scriptsize
    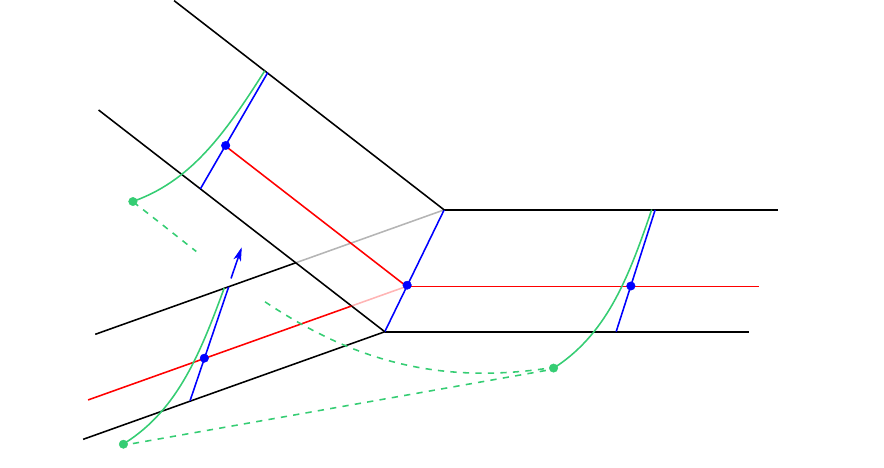}
    \caption{}
    \label{figure surj in building}
\end{figure}
\\For the rest of the proof, let us treat the building and symmetric space cases separately.

\medskip \noindent
\underline{$X$ is a Euclidean building:}
Let us show that there exists $z \in X_\Delta$ such that $\pi_1(z) = H_1'$ and for any $i \ne 1$ $\pi_i(z) = H_i$. By repeating the process, this completes the proof.
\begin{claim}\label{claim entering flat}
    $[x,\xi_1)$ enters in a maximal-flat $F'$ satisfying $\Delta \cup \{\xi_1'\} \subset \bound F'$.
\end{claim}
\begin{proof}[Proof of the claim]
    Let $H$ be the $(n-1)$-flat in $F$ satisfying $\pi_1(H) = \pi_1(x)$ (see \cref{fibers}), and let $s = \bound H$. Note that $\xi_2,\dots,\xi_n \in s$, therefore $s_1$ and $s$ both contain the top-dimensional cell that contains $\eta_1$. By \cite[Lemma 3.5]{leebcharac}, $CS(s)$ and $CS(s_1)$ have the same boundary, in particular $m_1,m_1' \in \partial CS(s)$. Since $\pi_1([x,\xi_1))$ is a parametrization of the geodesic ray $[\pi_1(x),m_1)$ in $CS(s)$, it enters in the geodesic joining $m_1'$ and $m_1$. Such a geodesic in $CS(s)$ corresponds to a maximal-flat $F'$ that contains $s$,$m_1$, and $m_1'$ in its boundary. By convexity, as in \cref{parallel sets are in X_Delta}, $\xi_1,\xi_1' \in \bound F'$.
\end{proof}
\noindent
Since $\pi_1(x) = \pi_1(x_1)$, it follows that $\pi_1([x,\xi_1)) = \pi_1([x_1,\xi_1))$. Let $p$ be the branching point of $[x_1,\xi_1)$ and $[x_1',\xi_1)$. There exists $y\in [x_1,\xi_1)$ such that $\pi_1(y) \in \pi_1([p,\xi_1))$, and $y$ is contained in a flat $F'$ that satisfies \cref{claim entering flat}. 

\medskip \noindent
Let us denote $\pi_1(y) = H_1''$, and let $x_1'' \in H_1'' \cap [x_1,\xi_1)$ such that $[y,\eta_1)$ is strongly asymptotic to $[x_1'',\eta_1)$. Now we will move backwards in $F'$ towards $\xi_1'$. $\pi_1(y) = \pi_1(x_1'')$, so $\pi_1([y,\xi_1')) = \pi_1([x_1'',\xi_1'))$. Since $x_1' \in [x_1'',\xi_1')$, there exists $z\in [y,\xi_1')$ such that $\pi_1(z) = \pi_1(x_1') = H_1'$, see \cref{figure surj in building}. $y \in F'$ and $\xi_1' \in \bound F_1'$ so $z \in F' \subset X_\Delta$. Moreover, in the path $x \xrightarrow[]{} y \xrightarrow[]{} z$, we followed geodesics pointing towards $\xi_1$ so $\pi_2, \dots, \pi_n$ are constant along the path.

\medskip \noindent
\underline{$X$ is a symmetric space:} we will show that for any $\varepsilon>0$, there exists $z \in X_\Delta$ such that $d_1(\pi_1(z),H_1') \leq \varepsilon$, and for any $i \ne 1$, $d_i(\pi_i(z),H_i) \leq \varepsilon$.

\medskip \noindent
We have the same setting except that $F_1$ and $F_1'$ do not share a singular half-space (if they do then they are equal). As in \cref{claim entering flat}, there exists a maximal-flat $F'$ such that $\Delta \cup \{\xi_1'\} \subset \bound F'$, and to which $[x,\xi_1)$ is strongly asymptotic. The proof is similar. As in the building case, $\pi_1([x,\xi_1)) = \pi_1([x_1,\xi_1))$. Note also that $[x_1,\xi_1)$ and $[x_1',\xi_1)$ are strongly asymptotic, so their images by $\pi_1$ are also strongly asymptotic ($\pi_1$ is $1$-Lipschitz). In particular $d\big(\pi_1([x,\xi_1)(t)),\pi_1([x_1',\xi_1))\big) \xrightarrow[+\infty]{}0$. To sum up
\begin{equation*}
    \begin{split}
            &[x,\xi_1) \textup{ is strongly asymptotic to } F',
           \\&\pi_1([x,\xi_1)) \textup{ is strongly asymptotic to } \pi_1([x_1',\xi_1)).
    \end{split}
\end{equation*}
So there exists $y \in [x,\xi_1)$ such that $y \in N_\varepsilon(F')$ and $d_1(\pi_1(y),\pi_1([x_1',\xi_1))) \leq \varepsilon$. Let $y' \in F'$ such that $d_X(y,y') \leq \varepsilon$, and let $x_1'' \in [x_1',\xi_1)$ such that $d_1(\pi_1(y),\pi_1(x_1'')) \leq \varepsilon$. So $d_1(\pi_1(y'),\pi_1(x_1'')) \leq 2 \varepsilon$.

\medskip \noindent
Now that $y' \in F'$ which is in $X_\Delta$ and contains $\xi_1'$ in its boundary, we can move towards it while staying in $X_\Delta$. $d_1(\pi_1(y'),\pi_1(x_1'')) \leq 2 \varepsilon$, so the geodesic rays $\pi_1([y',\xi_1'))$ and $\pi_1([x_1'',\xi_1'))$, in $CS(s_1)$, are at Hausdorff distance $\leq 2 \varepsilon$. $H_1' \in \pi_1([x_1'',\xi_1'))$ so there exists $z \in [y',\xi_1')$ such that $d_1(\pi_1(z),H_1') \leq 2 \varepsilon$. 

\medskip \noindent
Note that in both paths $x \xrightarrow[]{} y$ and $y' \xrightarrow[]{} z$, only $\pi_1$ changes so for any $i \ne 1$, $\pi_i(z) = \pi_i(y')$ which is at distance $\leq \varepsilon$ from $\pi_i(y) = \pi_i(x) = H_i$. So $z$ satisfies
\begin{equation*}
    \begin{split}
            &d_1(\pi_1(z),H_1') \leq 2 \varepsilon,
           \\&d_1(\pi_i(z),H_i) \leq \varepsilon, \textup{ for any } i \ne 1.
    \end{split}
\end{equation*}
By repeating the process for $i \ne 1$, we have shown that: for any $\varepsilon>0$, if $H_i \in CS(s_i)$ for $i = 1, \dots,n$, there exists $x' \in X_\Delta$ such that for any $i$, $d_i(\pi_i(x'),H_i)) \leq 2n \varepsilon$. By taking $\varepsilon$ small enough, this completes the proof of the quasi-surjectivity.
\begin{rem}
    Let us note that $\pi$ is also surjective for symmetric spaces, but the proof is tedious and the quasi-surjectivity is enough for our purpose.
\end{rem}
\begin{rem}
    Note also that the maximal distribution of the vertices $\xi_1,\dots, \xi_n$ is not needed for the surjectivity of $\pi$. We only needed that for any $i$, $\xi_i \not\in s_i$ and $\xi_i \in s_j$ for any $j \ne i$, so that one can move in the direction of $\xi_i$ and only changing $\pi_i$.
\end{rem}

\medskip \noindent
\textbf{Step $4$: $X_{\Delta} \to X$ is a bi-Lipschitz (resp.\ quasi-isometric) embedding:}

\medskip \noindent
Let us threat the two cases separately.

\medskip \noindent
\underline{$X$ is a Euclidean building:} let us prove a stronger result, from which the proof immediately follows.
\begin{prop}\label{path in union of flats}
    Let $\eta$ be an interior point of a chamber $C$ in $\bound X$. There exists a constant $\lambda>0$ such that for any maximal-flats $F_1$ and $F_2$ such that $\eta \in \bound F_1 \cap \bound F_2$, the following holds:
    \\ If $x \in F_1$ and $y \in F_2$, then there exists a path in $F_1 \cup F_2$ from $x$ to $y$ of length $\leq \lambda d_X(x,y)$.
\end{prop}
\noindent
Let us first prove the following lemma, which is the equivalent of \cref{lem convergence of geodesics} for buildings.
\begin{lem}\label{lem cv geodesics in building}
    Let $\eta$ be an interior point of a chamber $C$ in $\bound X$. There exists $\beta \geq 0$ such that the following holds:
    \\ For any $x,y \in X$, if $[x,\eta)$ is strongly asymptotic to $[y,\eta)$ and $b_{\eta}(x) = b_{\eta}(y)$, then
    $$       [x,\eta)(\beta  d) = [y,\eta)(\beta  d),$$
    where $d = d_X(x,y)$. In other words, the branching point of $[x,\eta)$ and $[y,\eta)$ is at distance $\leq \beta d_X(x,y)$ from $x$ and $y$.
\end{lem}
\begin{proof}
    Let $z$ be the branching point of $[x,\eta)$ and $[y,\eta)$. Note that  $b_{\eta}(x) = b_{\eta}(y)$ implies that $d_X(x,z) = d_X(y,z)$. Note also that, since $\eta$ is a regular point, $z$ is the entering point of $[x,\eta)$ in the cone $[y,C)$. Let $H$ be the wall in $F_2$ by which $[x,\eta)$ enters in $[y,C)$. Then 
    $$\angle_z(x,y) = 2 d_T(\eta, \bound H).$$ 
    Since $\eta$ is an interior point of $C$, 
    $$d_T(\eta, \bound H) \geq d_T(\eta,\partial C).$$
    Let $\theta = 2\, d_T(\eta,\partial C)$, which does not depend on $x$ and $y$. By considering the geodesic triangle $[x,y]$,$[y,z]$,$[z,x]$, (see \cite[Chap.2 Ex 1.9]{bridson2013metric})
    \begin{equation*}
        \begin{split}
            d_X(x,y)^2 &\geq d_X(x,z)^2+d_X(y,z)^2-2d_X(x,z)d_X(y,z) \cos(\angle_z(x,y))
                \\& \geq 2(1-\cos(\theta)) d_X(x,z)^2.
        \end{split}
    \end{equation*}
    Therefore,
    $$   d_X(x,z) \leq \frac{1}{\sqrt{2(1-\cos(\theta))}} \, d_X(x,y),  $$
    and the path $x \xrightarrow[]{} z \xrightarrow[]{} y$ has length $\leq \frac{2}{\sqrt{2(1-\cos(\theta))}} \, d_X(x,y).$ \qedhere
\end{proof}
\begin{proof}[Proof of \cref{path in union of flats}]
    Without loss of generality, suppose $b_{\eta}(x) - b_{\eta}(y) \geq 0$. 

    \medskip \noindent
    • If $[x,\eta)$ is strongly asymptotic to $[y,\eta)$, let $x' \in [x,\eta)$ such that $b_{\eta}(x') = b_{\eta}(y)$, and let $z$ be the branching point. By \cref{lem cv geodesics in building}, 
    $$d_X(x',z) = d_X(y,z) \leq \beta \, d_X(x',y).$$
    $x'$ is the projection of $x$ onto the horoball, centered at $\eta$, and whose boundary horocycle contains $x'$ and $y$, so
    $$d_X(x,x') \leq d_X(x,y).$$
    This projection is $1$-Lipschitz, so
    $$d_X(x',y) \leq d_X(x,y).$$
    Therefore, the path $x \xrightarrow[]{} x' \xrightarrow[]{} z \xrightarrow[]{} y$ has length 
    $$d_X(x,x')+d_X(x',z)+d_X(z,y) \leq d_X(x,y) + 2 \beta d_X(x',y) \leq (2 \beta + 1) d_X(x,y).$$
    \noindent
    \\• If not, let $y_1 \in F_2$ such that $[x,\eta)$ is strongly asymptotic to $[y_1,\eta)$ and $b_{\eta}(y_1) = b_{\eta}(y)$. Let $z$ be the branching point of $[x,\eta)$ and $[y_1,\eta)$, and $z' \in [y,\eta)$ such that $b_{\eta}(z') = b_{\eta}(z)$. We consider the path $x \xrightarrow[]{} z \xrightarrow[]{} y_1 \xrightarrow[]{} y$. By the first case,
    $$  d_X(x,z)+d_X(z,y_1) \leq (2 \beta + 1) d_X(x,y). $$ 
    $b_{\eta}(y_1) = b_{\eta}(y)$, $b_{\eta}(z') = b_{\eta}(z)$, and they are all in $F_2$, so 
    $$d_X(y,y_1) = d_X(z',z) = d_X(p(y),p(x)) \leq d_X(x,y),$$
    where $p$ is the projection onto the horoball centered at $\eta$, and whose boundary horocycle contains $z'$ and $z$. We conclude that this path has length $\leq (2 \beta + 2) d_X(x,y)$. \qedhere
\end{proof}

\medskip \noindent
\underline{$X$ is a symmetric space:} let us prove that for some $\delta>0$, $N_{\delta}(X_{\Delta})$, equipped with the path metric, embeds quasi-isometrically in $X$. To do so, as in the building case, let us prove the following result, which can be seen as a generalization of \cref{cor quasi path in rank one}, and from which the proof follows.
\begin{prop}\label{path in union of flats symmetric space}
    Let $\eta$ be an interior point of a chamber $C$ in $\bound X$. There exist constants $\delta,\lambda,K > 0$ such that for any maximal-flats $F_1$ and $F_2$ such that $\eta \in \bound F_1 \cap \bound F_2$, the following holds: if $x \in F_1$ and $y \in F_2$, then there exists a path in $N_{\delta}(F_1 \cup F_2)$ from $x$ to $y$ of length $\leq \lambda d_X(x,y) + K$.
\end{prop}
\begin{proof}
    Let $\delta$ be the constant of \cref{lem convergence of geodesics} and, without loss of generality, we suppose $b_{\eta}(x) - b_{\eta}(y) \geq 0$.
    
    \medskip \noindent
    • If $[x,\eta)$ is strongly asymptotic to $[y,\eta)$, \cref{lem convergence of geodesics} implies that there exists a path in $N_{\delta}([x,\eta) \cup [y,\eta)) \subset N_{\delta}(F_1 \cup F_2)$ of length $\leq 3 \, d_X(x,y)+ \delta$.
    \begin{figure}[!ht]
    \centering
    \def\svgwidth{0.57\textwidth}
    {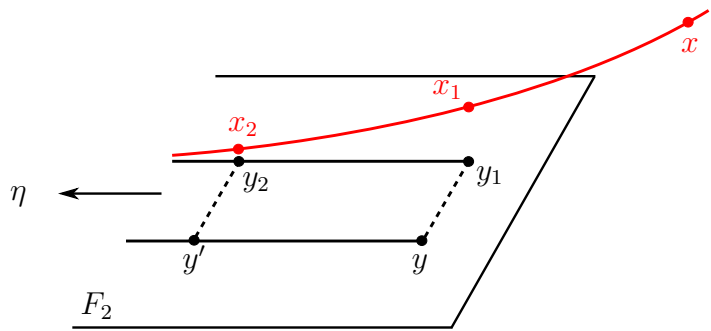}
    \caption{Path from $x$ to $y$ in $N_\delta(X_{\Delta})$.}
    \label{path in N(X_Delta)}
\end{figure}
    
    \medskip \noindent
    • If not, let us consider the following points: $y_1 \in F_2$ such that $[x,\eta)$ is strongly asymptotic to $[y_1,\eta)$ and $b_{\eta}(y_1) = b_{\eta}(y)$, $x_1 \in [x,\eta)$ such that $b_{\eta}(x_1) = b_{\eta}(y_1)$, $x_2 \in [x_1,\eta)$ such that $d_X(x_1,x_2) = d_X(x_1,y_1)$, $y_2 \in [y_1,\eta)$ such that $d_X(y_1,y_2) = d_X(x_1,y_1)$, and $y' \in [y,\eta)$ such that $b_{\eta}(y') = b_{\eta}(x_2) = b_{\eta}(y_2)$, see \cref{path in N(X_Delta)}.

    \medskip \noindent
    The first case implies that the path $x \xrightarrow[]{} x_1 \xrightarrow[]{} x_2 \xrightarrow[]{} y_2 \xrightarrow[]{} y_1$, which is in $N_{\delta}(F_1 \cup F_2)$, has length $\leq 3 d_X(x,y_1) + \delta$. And, as in the building case, $d_X(x,y_1) \leq d_X(x,y)$, and $d_X(y_1,y) = d_X(y_2,y') \leq d_X(x,y)$. We conclude that the path $x \xrightarrow[]{} x_1 \xrightarrow[]{} x_2 \xrightarrow[]{} y_2 \xrightarrow[]{} y_1 \xrightarrow[]{} y$ has length $\leq 4 d_X(x,y) + \delta $.
\end{proof}
\noindent
This completes the proof of \cref{mainthm general}.
\section{Appendix}
In this appendix, we show that the quasi-isometric embedding of the product of $n$ copies of $\mathbb{H}_{\mathbb{R}}^2$ into any symmetric space of non-compact type of rank $n$ can also be obtained as an $AN$-map. The idea of the proof was communicated to the first author by Yves Benoist. 

\medskip \noindent
Let us recall the following theorem due to Fisher--Whyte \cite[Theorem 1.5]{fisher2018quasi}.
\begin{theorem*}
    Let $G_1$ and $G_2$ be semisimple Lie groups of equal rank with Iwasawa decompositions $G_i = K_i A_i N_i$. Every injective homomorphism $A_1 N_1 \to A_2 N_2$ is a quasi-isometric embedding.
\end{theorem*}
\noindent
Let $X$ be a symmetric space of non-compact type of rank $n$, and $G = KAN$ an Iwasawa decomposition, where $G = \textup{Isom}_0(X)$. To show that there exists a quasi-isometric embedding from the product of $n$ copies of the real hyperbolic plane into $X$, we need to show that there exists a subgroup of $AN$ isomorphic to the product of $n$ copies of the affine group 
$$ \left\{ \begin{pmatrix}
e^t & s e^t \\
0 & e^{-t} 
\end{pmatrix} \Bigl\vert \, t,s \in \mathbb{R}  \right\}. $$
We use the notations from \cref{section Iwasawa}: Let $\mathfrak{a}$ be a Cartan subspace such that $A = \exp(\mathfrak{a})$, and $\mathfrak{n} = \bigoplus_{\alpha \in \Phi, \alpha(Y)>0} \mathfrak{g}_{\alpha}$ the sum of positive root spaces with respect to some regular vector $Y \in \mathfrak{a}$ such that $N = \exp(\mathfrak{n})$. Let us show that it is enough to find linearly independent positive roots $\alpha_1, \dots, \alpha_n$ such that for any $i \ne j$, $\alpha_i + \alpha_j$ is not a root.

\medskip \noindent
For any $i = 1, \dots , n$, take $Z_i \in \mathfrak{g}_{\alpha_i} \backslash \{0\} $. Since for any $i \ne j$, $\alpha_i + \alpha_j$ is not a root, $e^{Z_i}$ and $e^{Z_j}$ commute. Therefore 
$$\displaystyle \prod_{i = 1}^n \exp(\mathbb{R}.Z_i) \simeq \mathbb{R}^n,$$ 
on which $A$ acts diagonally. Indeed, since $\alpha_1, \dots, \alpha_n$ are linearly independent, pick, for any $i = 1, \dots, n$, $X_i \in \cap_{j \ne i} \ker \alpha_j$ such that $\alpha_i(X_i) = 2$. So for any $i \ne j$, $e^{X_i}$ and $e^{Z_j}$ commute:
$$e^{X_i}e^{Z_j} = e^{X_i}e^{Z_j} e^{-X_i}e^{X_i} = \exp(e^{\alpha_j(X_i)}Z_j)e^{X_i} = e^{Z_j}e^{X_i}.$$
Therefore, for any $t_1, \dots, t_n,s_1, \dots, s_n \in \mathbb{R}$,
$$\displaystyle H := \exp\left(\sum_{i=1}^n t_i X_i\right)\exp\left(\sum_{j=1}^n s_j Z_j\right) = \prod_{i = 1}^n e^{t_i X_i}e^{s_i Z_i}.$$
It is easy to check that for any $i = 1, \dots,n$, 
$$  \left\{ e^{t X_i}e^{s Z_i} \bigl\vert \, t,s \in \mathbb{R}  \right\} \simeq \left\{ \begin{pmatrix}
e^t & s e^t \\
0 & e^{-t} 
\end{pmatrix} \Bigl\vert \, t,s \in \mathbb{R}  \right\}. $$
Therefore $H$ is the desired subgroup of $AN$ which isomorphic to the product of $n$ copies of the affine group.

\noindent\medskip
\textbf{Finding the roots $\alpha_1, \dots, \alpha_n$:}
We recall that there exists a natural order on the set of positive roots: given two roots $\alpha$ and $\beta$, $\alpha \leq \beta$ iff $\beta-\alpha$ is a non-negative linear combination of simple roots. We refer to \cite[Chap.6]{bourbaki81} for more details. We start by taking $\alpha_1$ the biggest positive root, and let $\alpha_2$ be the biggest positive root than is not in $\textup{span}\{\alpha_1\}$. $\alpha_1 + \alpha_2$ is not a root, otherwise it would be bigger than $\alpha_1$. Take $\alpha_3$ the biggest positive root that is not in $\textup{span}\{\alpha_1, \alpha_2\}$. Again, and by the same argument, $\alpha_1 + \alpha_3$ and $\alpha_2 + \alpha_3$ are not roots. We conclude by induction.

\bibliographystyle{alpha}
{\footnotesize\bibliography{Biblio}}
\end{document}

%% file: A1xA1.pdf_tex
\begingroup%
  \makeatletter%
  \providecommand\color[2][]{%
    \errmessage{(Inkscape) Color is used for the text in Inkscape, but the package 'color.sty' is not loaded}%
    \renewcommand\color[2][]{}%
  }%
  \providecommand\transparent[1]{%
    \errmessage{(Inkscape) Transparency is used (non-zero) for the text in Inkscape, but the package 'transparent.sty' is not loaded}%
    \renewcommand\transparent[1]{}%
  }%
  \providecommand\rotatebox[2]{#2}%
  \newcommand*\fsize{\dimexpr\f@size pt\relax}%
  \newcommand*\lineheight[1]{\fontsize{\fsize}{#1\fsize}\selectfont}%
  \ifx\svgwidth\undefined%
    \setlength{\unitlength}{111.14371369bp}%
    \ifx\svgscale\undefined%
      \relax%
    \else%
      \setlength{\unitlength}{\unitlength * \real{\svgscale}}%
    \fi%
  \else%
    \setlength{\unitlength}{\svgwidth}%
  \fi%
  \global\let\svgwidth\undefined%
  \global\let\svgscale\undefined%
  \makeatother%
  \begin{picture}(1,0.98078539)%
    \lineheight{1}%
    \setlength\tabcolsep{0pt}%
    \put(0,0){\includegraphics[width=\unitlength,page=1]{A1xA1.pdf}}%
  \end{picture}%
\endgroup%

%% file: A2.pdf_tex
\begingroup%
  \makeatletter%
  \providecommand\color[2][]{%
    \errmessage{(Inkscape) Color is used for the text in Inkscape, but the package 'color.sty' is not loaded}%
    \renewcommand\color[2][]{}%
  }%
  \providecommand\transparent[1]{%
    \errmessage{(Inkscape) Transparency is used (non-zero) for the text in Inkscape, but the package 'transparent.sty' is not loaded}%
    \renewcommand\transparent[1]{}%
  }%
  \providecommand\rotatebox[2]{#2}%
  \newcommand*\fsize{\dimexpr\f@size pt\relax}%
  \newcommand*\lineheight[1]{\fontsize{\fsize}{#1\fsize}\selectfont}%
  \ifx\svgwidth\undefined%
    \setlength{\unitlength}{111.00976917bp}%
    \ifx\svgscale\undefined%
      \relax%
    \else%
      \setlength{\unitlength}{\unitlength * \real{\svgscale}}%
    \fi%
  \else%
    \setlength{\unitlength}{\svgwidth}%
  \fi%
  \global\let\svgwidth\undefined%
  \global\let\svgscale\undefined%
  \makeatother%
  \begin{picture}(1,0.94600625)%
    \lineheight{1}%
    \setlength\tabcolsep{0pt}%
    \put(0,0){\includegraphics[width=\unitlength,page=1]{A2.pdf}}%
  \end{picture}%
\endgroup%

%% file: A3.pdf_tex
\begingroup%
  \makeatletter%
  \providecommand\color[2][]{%
    \errmessage{(Inkscape) Color is used for the text in Inkscape, but the package 'color.sty' is not loaded}%
    \renewcommand\color[2][]{}%
  }%
  \providecommand\transparent[1]{%
    \errmessage{(Inkscape) Transparency is used (non-zero) for the text in Inkscape, but the package 'transparent.sty' is not loaded}%
    \renewcommand\transparent[1]{}%
  }%
  \providecommand\rotatebox[2]{#2}%
  \newcommand*\fsize{\dimexpr\f@size pt\relax}%
  \newcommand*\lineheight[1]{\fontsize{\fsize}{#1\fsize}\selectfont}%
  \ifx\svgwidth\undefined%
    \setlength{\unitlength}{226.23325364bp}%
    \ifx\svgscale\undefined%
      \relax%
    \else%
      \setlength{\unitlength}{\unitlength * \real{\svgscale}}%
    \fi%
  \else%
    \setlength{\unitlength}{\svgwidth}%
  \fi%
  \global\let\svgwidth\undefined%
  \global\let\svgscale\undefined%
  \makeatother%
  \begin{picture}(1,0.97618096)%
    \lineheight{1}%
    \setlength\tabcolsep{0pt}%
    \put(0,0){\includegraphics[width=\unitlength,page=1]{A3.pdf}}%
  \end{picture}%
\endgroup%

%% file: A3_with_points.pdf_tex
\begingroup%
  \makeatletter%
  \providecommand\color[2][]{%
    \errmessage{(Inkscape) Color is used for the text in Inkscape, but the package 'color.sty' is not loaded}%
    \renewcommand\color[2][]{}%
  }%
  \providecommand\transparent[1]{%
    \errmessage{(Inkscape) Transparency is used (non-zero) for the text in Inkscape, but the package 'transparent.sty' is not loaded}%
    \renewcommand\transparent[1]{}%
  }%
  \providecommand\rotatebox[2]{#2}%
  \newcommand*\fsize{\dimexpr\f@size pt\relax}%
  \newcommand*\lineheight[1]{\fontsize{\fsize}{#1\fsize}\selectfont}%
  \ifx\svgwidth\undefined%
    \setlength{\unitlength}{274.10975612bp}%
    \ifx\svgscale\undefined%
      \relax%
    \else%
      \setlength{\unitlength}{\unitlength * \real{\svgscale}}%
    \fi%
  \else%
    \setlength{\unitlength}{\svgwidth}%
  \fi%
  \global\let\svgwidth\undefined%
  \global\let\svgscale\undefined%
  \makeatother%
  \begin{picture}(1,0.80567944)%
    \lineheight{1}%
    \setlength\tabcolsep{0pt}%
    \put(0,0){\includegraphics[width=\unitlength,page=1]{A3_with_points.pdf}}%
    \put(-0.00242262,0.4526562){\color[rgb]{1,0,0}\makebox(0,0)[lt]{\lineheight{1.25}\smash{\begin{tabular}[t]{l}$\xi_1$\end{tabular}}}}%
    \put(0.73933402,0.03335507){\color[rgb]{1,0,0}\makebox(0,0)[lt]{\lineheight{1.25}\smash{\begin{tabular}[t]{l}$\xi_2$\end{tabular}}}}%
    \put(0.71860476,0.77120741){\color[rgb]{1,0,0}\makebox(0,0)[lt]{\lineheight{1.25}\smash{\begin{tabular}[t]{l}$\xi_3$\end{tabular}}}}%
    \put(0,0){\includegraphics[width=\unitlength,page=2]{A3_with_points.pdf}}%
    \put(0.8750073,0.22192523){\color[rgb]{0,0,0}\makebox(0,0)[lt]{\lineheight{1.25}\smash{\begin{tabular}[t]{l}$\eta_1$\end{tabular}}}}%
    \put(0.17714923,0.62610231){\color[rgb]{0,0,0}\makebox(0,0)[lt]{\lineheight{1.25}\smash{\begin{tabular}[t]{l}$\eta_2$\end{tabular}}}}%
    \put(0.22811486,0.20609907){\color[rgb]{0,0,0}\makebox(0,0)[lt]{\lineheight{1.25}\smash{\begin{tabular}[t]{l}$\eta_3$\end{tabular}}}}%
  \end{picture}%
\endgroup%

%% file: flat_intersection.pdf_tex
\begingroup%
  \makeatletter%
  \providecommand\color[2][]{%
    \errmessage{(Inkscape) Color is used for the text in Inkscape, but the package 'color.sty' is not loaded}%
    \renewcommand\color[2][]{}%
  }%
  \providecommand\transparent[1]{%
    \errmessage{(Inkscape) Transparency is used (non-zero) for the text in Inkscape, but the package 'transparent.sty' is not loaded}%
    \renewcommand\transparent[1]{}%
  }%
  \providecommand\rotatebox[2]{#2}%
  \newcommand*\fsize{\dimexpr\f@size pt\relax}%
  \newcommand*\lineheight[1]{\fontsize{\fsize}{#1\fsize}\selectfont}%
  \ifx\svgwidth\undefined%
    \setlength{\unitlength}{241.43846948bp}%
    \ifx\svgscale\undefined%
      \relax%
    \else%
      \setlength{\unitlength}{\unitlength * \real{\svgscale}}%
    \fi%
  \else%
    \setlength{\unitlength}{\svgwidth}%
  \fi%
  \global\let\svgwidth\undefined%
  \global\let\svgscale\undefined%
  \makeatother%
  \begin{picture}(1,0.57043696)%
    \lineheight{1}%
    \setlength\tabcolsep{0pt}%
    \put(0,0){\includegraphics[width=\unitlength,page=1]{flat_intersection.pdf}}%
    \put(0.55006572,0.31312024){\color[rgb]{1,0,0}\makebox(0,0)[lt]{\lineheight{1.25}\smash{\begin{tabular}[t]{l}$F_x\cap F_y$\end{tabular}}}}%
    \put(0.89960978,0.10863639){\color[rgb]{0,0,1}\makebox(0,0)[lt]{\lineheight{1.25}\smash{\begin{tabular}[t]{l}$H$\end{tabular}}}}%
    \put(0.01894894,0.24402422){\color[rgb]{0.15686275,0.6745098,0.23529412}\makebox(0,0)[lt]{\lineheight{1.25}\smash{\begin{tabular}[t]{l}$M$\end{tabular}}}}%
    \put(0.20553055,0.01515957){\color[rgb]{0,0,0}\makebox(0,0)[lt]{\lineheight{1.25}\smash{\begin{tabular}[t]{l}$x$\end{tabular}}}}%
    \put(0.92904669,0.04692988){\color[rgb]{0,0,0}\makebox(0,0)[lt]{\lineheight{1.25}\smash{\begin{tabular}[t]{l}$\eta_1$\end{tabular}}}}%
    \put(0.44535254,0.00781449){\color[rgb]{0,0,0}\makebox(0,0)[lt]{\lineheight{1.25}\smash{\begin{tabular}[t]{l}$[x,\eta_1)$\end{tabular}}}}%
  \end{picture}%
\endgroup%

%% file: quasi_path_rank_1.pdf_tex
\begingroup%
  \makeatletter%
  \providecommand\color[2][]{%
    \errmessage{(Inkscape) Color is used for the text in Inkscape, but the package 'color.sty' is not loaded}%
    \renewcommand\color[2][]{}%
  }%
  \providecommand\transparent[1]{%
    \errmessage{(Inkscape) Transparency is used (non-zero) for the text in Inkscape, but the package 'transparent.sty' is not loaded}%
    \renewcommand\transparent[1]{}%
  }%
  \providecommand\rotatebox[2]{#2}%
  \newcommand*\fsize{\dimexpr\f@size pt\relax}%
  \newcommand*\lineheight[1]{\fontsize{\fsize}{#1\fsize}\selectfont}%
  \ifx\svgwidth\undefined%
    \setlength{\unitlength}{189.79579945bp}%
    \ifx\svgscale\undefined%
      \relax%
    \else%
      \setlength{\unitlength}{\unitlength * \real{\svgscale}}%
    \fi%
  \else%
    \setlength{\unitlength}{\svgwidth}%
  \fi%
  \global\let\svgwidth\undefined%
  \global\let\svgscale\undefined%
  \makeatother%
  \begin{picture}(1,1.0396844)%
    \lineheight{1}%
    \setlength\tabcolsep{0pt}%
    \put(0,0){\includegraphics[width=\unitlength,page=1]{quasi_path_rank_1.pdf}}%
    \put(0.51506161,1.01566599){\color[rgb]{0,0,1}\makebox(0,0)[lt]{\lineheight{1.25}\smash{\begin{tabular}[t]{l}$\eta$\end{tabular}}}}%
    \put(0.83418076,0.2209021){\color[rgb]{0,0,0}\makebox(0,0)[lt]{\lineheight{1.25}\smash{\begin{tabular}[t]{l}$x$\end{tabular}}}}%
    \put(0.77899483,0.35848413){\color[rgb]{0,0,0}\makebox(0,0)[lt]{\lineheight{1.25}\smash{\begin{tabular}[t]{l}$x'$\end{tabular}}}}%
    \put(0.2060019,0.30137443){\color[rgb]{0,0,0}\makebox(0,0)[lt]{\lineheight{1.25}\smash{\begin{tabular}[t]{l}$y$\end{tabular}}}}%
    \put(0.57612389,0.68255337){\color[rgb]{0,0,0}\makebox(0,0)[lt]{\lineheight{1.25}\smash{\begin{tabular}[t]{l}$x''$\end{tabular}}}}%
    \put(0.35351645,0.6832228){\color[rgb]{0,0,0}\makebox(0,0)[lt]{\lineheight{1.25}\smash{\begin{tabular}[t]{l}$y''$\end{tabular}}}}%
    \put(0.44907796,0.6208327){\color[rgb]{1,0,0}\makebox(0,0)[lt]{\lineheight{1.25}\smash{\begin{tabular}[t]{l}\tiny$\leq\delta$\end{tabular}}}}%
    \put(0,0){\includegraphics[width=\unitlength,page=2]{quasi_path_rank_1.pdf}}%
  \end{picture}%
\endgroup%

%% file: surj_in_building.pdf_tex
\begingroup%
  \makeatletter%
  \providecommand\color[2][]{%
    \errmessage{(Inkscape) Color is used for the text in Inkscape, but the package 'color.sty' is not loaded}%
    \renewcommand\color[2][]{}%
  }%
  \providecommand\transparent[1]{%
    \errmessage{(Inkscape) Transparency is used (non-zero) for the text in Inkscape, but the package 'transparent.sty' is not loaded}%
    \renewcommand\transparent[1]{}%
  }%
  \providecommand\rotatebox[2]{#2}%
  \newcommand*\fsize{\dimexpr\f@size pt\relax}%
  \newcommand*\lineheight[1]{\fontsize{\fsize}{#1\fsize}\selectfont}%
  \ifx\svgwidth\undefined%
    \setlength{\unitlength}{424.29659489bp}%
    \ifx\svgscale\undefined%
      \relax%
    \else%
      \setlength{\unitlength}{\unitlength * \real{\svgscale}}%
    \fi%
  \else%
    \setlength{\unitlength}{\svgwidth}%
  \fi%
  \global\let\svgwidth\undefined%
  \global\let\svgscale\undefined%
  \makeatother%
  \begin{picture}(1,0.52849437)%
    \lineheight{1}%
    \setlength\tabcolsep{0pt}%
    \put(0,0){\includegraphics[width=\unitlength,page=1]{surj_in_building.pdf}}%
    \put(0.71682531,0.17997412){\color[rgb]{0,0,1}\makebox(0,0)[lt]{\lineheight{1.25}\smash{\begin{tabular}[t]{l}$x''_1$\end{tabular}}}}%
    \put(0.27177995,0.36341332){\color[rgb]{0,0,1}\makebox(0,0)[lt]{\lineheight{1.25}\smash{\begin{tabular}[t]{l}$x_1$\end{tabular}}}}%
    \put(0.2360684,0.09914787){\color[rgb]{0,0,1}\makebox(0,0)[lt]{\lineheight{1.25}\smash{\begin{tabular}[t]{l}$x'_1$\end{tabular}}}}%
    \put(0.6931762,0.12414587){\color[rgb]{0,0,1}\makebox(0,0)[lt]{\lineheight{1.25}\smash{\begin{tabular}[t]{l}$H''_1$\end{tabular}}}}%
    \put(0.20866765,0.28574076){\color[rgb]{0,0,1}\makebox(0,0)[lt]{\lineheight{1.25}\smash{\begin{tabular}[t]{l}$H_1$\end{tabular}}}}%
    \put(0.2077749,0.04915173){\color[rgb]{0,0,1}\makebox(0,0)[lt]{\lineheight{1.25}\smash{\begin{tabular}[t]{l}$H'_1$\end{tabular}}}}%
    \put(0,0){\includegraphics[width=\unitlength,page=2]{surj_in_building.pdf}}%
    \put(0.32606649,0.49911721){\color[rgb]{0,0,1}\makebox(0,0)[lt]{\lineheight{1.25}\smash{\begin{tabular}[t]{l}$\eta_1$\end{tabular}}}}%
    \put(0.74976962,0.3473431){\color[rgb]{0,0,1}\makebox(0,0)[lt]{\lineheight{1.25}\smash{\begin{tabular}[t]{l}$\eta_1$\end{tabular}}}}%
    \put(0.27078852,0.25449306){\color[rgb]{0,0,1}\makebox(0,0)[lt]{\lineheight{1.25}\smash{\begin{tabular}[t]{l}$\eta_1$\end{tabular}}}}%
    \put(0.19294803,0.4895964){\color[rgb]{0,0,0}\makebox(0,0)[lt]{\lineheight{1.25}\smash{\begin{tabular}[t]{l}$F_1$\end{tabular}}}}%
    \put(0.1300965,0.13113087){\color[rgb]{0,0,0}\makebox(0,0)[lt]{\lineheight{1.25}\smash{\begin{tabular}[t]{l}$F'_1$\end{tabular}}}}%
    \put(0.46073959,0.1791095){\color[rgb]{0,0,1}\makebox(0,0)[lt]{\lineheight{1.25}\smash{\begin{tabular}[t]{l}$p$\end{tabular}}}}%
    \put(0.12362706,0.27822311){\color[rgb]{0.20784314,0.80392157,0.45098039}\makebox(0,0)[lt]{\lineheight{1.25}\smash{\begin{tabular}[t]{l}$x$\end{tabular}}}}%
    \put(0.62675036,0.08547095){\color[rgb]{0.20784314,0.80392157,0.45098039}\makebox(0,0)[lt]{\lineheight{1.25}\smash{\begin{tabular}[t]{l}$y$\end{tabular}}}}%
    \put(0.12108149,0.00277113){\color[rgb]{0.20784314,0.80392157,0.45098039}\makebox(0,0)[lt]{\lineheight{1.25}\smash{\begin{tabular}[t]{l}$z$\end{tabular}}}}%
    \put(0,0){\includegraphics[width=\unitlength,page=3]{surj_in_building.pdf}}%
    \put(0.00048795,0.04843101){\color[rgb]{0,0,0}\makebox(0,0)[lt]{\lineheight{1.25}\smash{\begin{tabular}[t]{l}$\xi'_1$\end{tabular}}}}%
    \put(0.93340081,0.1986797){\color[rgb]{0,0,0}\makebox(0,0)[lt]{\lineheight{1.25}\smash{\begin{tabular}[t]{l}$\xi_1$\end{tabular}}}}%
    \put(0.93275093,0.15631935){\color[rgb]{0,0,0}\makebox(0,0)[lt]{\lineheight{1.25}\smash{\begin{tabular}[t]{l}$m_1$\end{tabular}}}}%
    \put(-0.00156509,0.10890935){\color[rgb]{0,0,0}\makebox(0,0)[lt]{\lineheight{1.25}\smash{\begin{tabular}[t]{l}$m'_1$\end{tabular}}}}%
    \put(0.78868757,0.25895706){\color[rgb]{0,0,0}\makebox(0,0)[lt]{\lineheight{1.25}\smash{\begin{tabular}[t]{l}$F_1\cap F'_1$\end{tabular}}}}%
  \end{picture}%
\endgroup%

%% file: path_in_NX_Delta.pdf_tex
\begingroup%
  \makeatletter%
  \providecommand\color[2][]{%
    \errmessage{(Inkscape) Color is used for the text in Inkscape, but the package 'color.sty' is not loaded}%
    \renewcommand\color[2][]{}%
  }%
  \providecommand\transparent[1]{%
    \errmessage{(Inkscape) Transparency is used (non-zero) for the text in Inkscape, but the package 'transparent.sty' is not loaded}%
    \renewcommand\transparent[1]{}%
  }%
  \providecommand\rotatebox[2]{#2}%
  \newcommand*\fsize{\dimexpr\f@size pt\relax}%
  \newcommand*\lineheight[1]{\fontsize{\fsize}{#1\fsize}\selectfont}%
  \ifx\svgwidth\undefined%
    \setlength{\unitlength}{388.64375471bp}%
    \ifx\svgscale\undefined%
      \relax%
    \else%
      \setlength{\unitlength}{\unitlength * \real{\svgscale}}%
    \fi%
  \else%
    \setlength{\unitlength}{\svgwidth}%
  \fi%
  \global\let\svgwidth\undefined%
  \global\let\svgscale\undefined%
  \makeatother%
  \begin{picture}(1,0.45169131)%
    \lineheight{1}%
    \setlength\tabcolsep{0pt}%
    \put(0,0){\includegraphics[width=\unitlength,page=1]{path_in_NX_Delta.pdf}}%
    \put(0.56477752,0.08611476){\color[rgb]{0,0,0}\makebox(0,0)[lt]{\lineheight{1.25}\smash{\begin{tabular}[t]{l}$y$\end{tabular}}}}%
    \put(0,0){\includegraphics[width=\unitlength,page=2]{path_in_NX_Delta.pdf}}%
    \put(0.24062841,0.0853553){\color[rgb]{0,0,0}\makebox(0,0)[lt]{\lineheight{1.25}\smash{\begin{tabular}[t]{l}$y'$\end{tabular}}}}%
    \put(0.65553287,0.21079155){\color[rgb]{0,0,0}\makebox(0,0)[lt]{\lineheight{1.25}\smash{\begin{tabular}[t]{l}$y_1$\end{tabular}}}}%
    \put(0.32553908,0.20210748){\color[rgb]{0,0,0}\makebox(0,0)[lt]{\lineheight{1.25}\smash{\begin{tabular}[t]{l}$y_2$\end{tabular}}}}%
    \put(0.94500105,0.39122671){\color[rgb]{1,0,0}\makebox(0,0)[lt]{\lineheight{1.25}\smash{\begin{tabular}[t]{l}$x$\end{tabular}}}}%
    \put(0.59377954,0.3314259){\color[rgb]{1,0,0}\makebox(0,0)[lt]{\lineheight{1.25}\smash{\begin{tabular}[t]{l}$x_1$\end{tabular}}}}%
    \put(0.30527633,0.27932166){\color[rgb]{1,0,0}\makebox(0,0)[lt]{\lineheight{1.25}\smash{\begin{tabular}[t]{l}$x_2$\end{tabular}}}}%
    \put(0.09685923,0.02070741){\color[rgb]{0,0,0}\makebox(0,0)[lt]{\lineheight{1.25}\smash{\begin{tabular}[t]{l}$F_2$\end{tabular}}}}%
    \put(0,0){\includegraphics[width=\unitlength,page=3]{path_in_NX_Delta.pdf}}%
    \put(-0.002563,0.18219291){\color[rgb]{0,0,0}\makebox(0,0)[lt]{\lineheight{1.25}\smash{\begin{tabular}[t]{l}$\eta$\end{tabular}}}}%
  \end{picture}%
\endgroup%